\documentclass[10pt,leqno, final]{amsart}
\usepackage[colorlinks,linkcolor=blue,citecolor=blue,urlcolor=blue]{hyperref}
\usepackage{amsfonts,amssymb}
\usepackage{color}
\usepackage{graphicx}
\usepackage{mathtools}
\usepackage{extarrows}
\usepackage{algorithmic}   
\usepackage{caption}
\usepackage{subcaption}
\usepackage{tikz}
\usepackage{enumitem}
\usepackage[obeyFinal]{todonotes}
\usepackage{comment}

\usepackage{algorithmic} 

\usepackage[notref,notcite]{showkeys}

\newtheorem{theorem}{Theorem}[section]
\newtheorem{lemma}[theorem]{Lemma}
\newtheorem{corollary}[theorem]{Corollary}

\newtheorem{assumption}[theorem]{Assumption}

\theoremstyle{definition}
\newtheorem{definition}[theorem]{Definition}
\newtheorem{example}[theorem]{Example}
\newtheorem*{example*}{Example}
\newtheorem{remark}[theorem]{Remark}
\newtheorem{algorithm}[theorem]{Algorithm}

\newcommand{\DD}{\mathcal{D}}

\newcommand{\FF}{\mathcal{F}}

\newcommand{\JJ}{\mathcal{J}}

\newcommand{\PP}{\mathcal{P}}

\newcommand{\UU}{\mathcal{U}}

\newcommand{\XX}{\mathcal{X}}
\newcommand{\YY}{\mathcal{Y}}
\newcommand{\ZZ}{\mathcal{Z}}

\newcommand{\N}{\mathbb{N}}

\newcommand{\R}{\mathbb{R}}

\newcommand{\W}{\mathbb{W}}

\newcommand{\frakM}{\mathfrak{M}}

\newcommand{\embed}{\hookrightarrow}

\newcommand{\weak}{\rightharpoonup}

\newcommand{\conv}{\operatorname{conv}}

\newcommand{\dom}{\operatorname{dom}}

\newcommand{\interior}{\operatorname{int}}

\renewcommand{\d}{\mathrm{d}}

\newcommand{\dual}[2]{\langle #1 , #2 \rangle}

\newcommand{\TV}{\mathrm{TV}}
\newcommand{\BV}{\mathrm{BV}}

\begin{document}

	\title[Switching Point Optimization for Abstract Parabolic Equations]
	{Switching Point Optimization for Abstract Parabolic Equations}\thanks{
		The authors acknowledge funding by Deutsche
		Forschungsgemeinschaft (DFG) under grant no.\ ME 3281/12-1.}
	
	\author{Christoph Buchheim} \address{Technische Universit\"at Dortmund, Fakult\"at f\"ur
		Mathematik, Lehrstuhl LSV, Vogelpothsweg 87, 44227 Dortmund, Germany}
	\email{christoph.buchheim@tu-dortmund.de}

	\author{Christian Meyer} \address{Technische Universit\"at Dortmund, Fakult\"at f\"ur
		Mathematik, Lehrstuhl LSX, Vogelpothsweg 87, 44227 Dortmund, Germany}
	\email{christian2.meyer@tu-dortmund.de}

	\author{Alimhan Musalatov} \address{Technische Universit\"at Dortmund, Fakult\"at f\"ur
		Mathematik, Lehrstuhl LSX, Vogelpothsweg 87, 44227 Dortmund, Germany}
	\email{alimhan.musalatov@tu-dortmund.de}

	\subjclass[2010]{49K20, 49M37, 90C30, 90C26} 
	
	\date{\today} 
	
	\keywords{Switching point optimization, maximal parabolic regularity, proximal gradient method, 
	convexification, extended formulations}
	
	\begin{abstract} 
		This work is concerned with a switching point optimization problem governed by a
		semilinear parabolic equation in abstract function spaces. It is shown that the 
		switching-point-to-control mapping is continuously Fr\'echet-differentiable when 
		considered with values in the dual of H\"older continuous functions in time. 
		By treating the state equation in weak form based on the concept of maximal parabolic regularity, 
		one can then show that the 
		reduced objective is continuously differentiable w.r.t.\ the switching points, which allows 
		to use gradient-based methods like the proximal gradient method for its minimization. 
		Numerical experiments confirm our theoretical findings, but also illustrate 
		that such a method will in general not be able to solve the problem to global 
		optimality due to the non-convex nature of the switching-point-to-control map.
		We therefore give a precise characterization of the convex hull of set of feasible switching 
		functions in terms of an extended formulation. The latter might be useful for a branch-and-bound 
		approach for the computation of global minimizers, but this is subject to future research.
	\end{abstract}
	
	\maketitle
	
	\section{Introduction}
	
    In this work, we investigate an optimal control problem governed by an abstract parabolic equation, 
	where the control is given in terms of a switching function in time. To be more precise, the state equation under 
	consideration reads
	\begin{equation}\label{eq:stateeq}
		y'(t) + A y(t) + f(y(t)) = \sum_{i=1}^{n-1}\chi_{[\tau_i, \tau_{i+1})}(t)\,\psi_i, 
		\quad t \in (0,T),\quad y(0) = y_0,
	\end{equation}
	where $T>0$ is a given final time, $A:X \supset \DD \to X$ is a linear operator in reflexive Banach spaces 
	$\DD$ and $X$, and $f: \DD \to X$ is a nonlinear mapping. 
	Moreover, $\chi_{[\tau_i, \tau_{i+1})}$ denotes the characteristic function of the interval $[\tau_i, \tau_{i+1})$, i.e., 
	\begin{equation*}
	    \chi_{[\tau_i, \tau_{i+1})}(t)
	    := 
	    \begin{cases}
	        1, & t \in [\tau_i, \tau_{i+1}),\\
	        0, & \text{else}
	    \end{cases}
	\end{equation*}
	and $\psi_i\in X$, $i = 1, \ldots, n-1$, are given functions.
	The optimization variables are the switching times $\tau_1, \ldots, \tau_n \in \R$.
	Formally, the optimization problem under consideration reads
	\begin{equation}\label{eq:ocp0}
		\left\{\quad
		\begin{aligned}
			\min \quad & J(y,\tau) \\
			\text{s.t.} \quad & \text{$y$ solves \eqref{eq:stateeq}},\quad
			\tau \in \PP.
		\end{aligned}
		\right.
	\end{equation}   
    Herein, $J$ is a given objective and $\PP$ denotes the set of feasible switching points, which is defined by
	\begin{equation}\label{eq:feaspoly}
		\PP := \{ \tau = (\tau_1, \ldots, \tau_n) \colon 0 \leq \tau_1 \leq \tau_2 \leq \ldots\leq \tau_n \leq T\}.
	\end{equation}
    The precise assumptions on the data as well as a rigorous statement of the optimization problem
    including a precise notion of solutions to \eqref{eq:stateeq} will be given in Section~\ref{sec:assus} below.

	\begin{remark}
		The restrictions in $\PP$ model that only one inhomogeneity can be active at a time. 
		This is motivated by the following example: Suppose that 
		\begin{equation}\label{eq:switch}
			\psi_i :=
			\begin{cases}
				0, & i \text{ even},\\
				\psi, & i \text{ odd},
			\end{cases}
		\end{equation}
		with a given form function $\psi$. Then \eqref{eq:ocp0} corresponds to a switching process, 
		where the source $\psi$ can be switched on resp.\ off for $n$ times and these switching times 
		represent the optimization variables.
	\end{remark}
	
	Though finite dimensional, optimizing \eqref{eq:ocp0} -- even only locally -- is a non-trivial problem
	because of the lack of differentiability of the solution operator of \eqref{eq:stateeq} mapping $\tau$ to $y$. 
    This turns the problem into a non-smooth and non-convex optimization problem, even if $f$
    in \eqref{eq:stateeq} is linear-affine and the objective~$J(y,\tau)$ in \eqref{eq:ocp0} is convex in both~$y$ and~$\tau$. 
    We overcome this issue by employing a weak solution concept 
    for the state equation along with a modification of the switching-point-to-control operator 
    $\tau \mapsto \sum_{i=1}^{n-1}\chi_{[\tau_i, \tau_{i+1})}\,\psi_i$ outside the feasible set $\PP$. 
    This allows us to use a standard proximal gradient method for the solution of \eqref{eq:ocp0}. 
    However, our numerical experiments show that, in general, the algorithm converges to stationary points
    that are no global minimizers, even if the state equation is linear and the objective is quadratic and strictly convex. 
    This underlines the non-convex nature of the problem under consideration. 
    A global optimization of \eqref{eq:ocp0} therefore requires other methods such as 
    branch-and-bound algorithms that build on the computation of dual bounds based on 
    tailored convexifications. 
    In preparation of the latter, we provide 
    a precise characterization of the convex hull of the set of feasible switching functions by means of an extended formulation. 
        
    Let us put our work into perspective. There are numerous contributions to switching point optimization 
    problems, in particular in the context of optimal control of ordinary differential equations (ODEs).
    Especially in the engineering literature, 
    it is frequently claimed that the operator mapping the switching times $\tau$ to the state and the objective, respectively, 
    is differentiable. This however is a delicate issue, even in the ODE case. The differentiability properties of 
    these mappings strongly depend on the underlying functional analytical setting, in particular the choice of state space. 
    If the latter  
    is chosen to be continuous in time, as for instance in case of objectives containing point evaluations of the state,
    then the mapping of $\tau$ to the (reduced) objective may fail to be differentiable. This concerns point evaluations 
    at fixed, prescribed time points as well as point evaluations at the switching points $\tau_k$. 
    To illustrate this aspect, we include an elementary ODE counterexample in Appendix~\ref{sec:nodiff}, 
    see, in particular, Remark~\ref{rem:nodiff}.
    
    Differentiability issues for switching point optimization problems have been addressed by several authors, 
    we only mention \cite{EWD03, EWA06}, where the directional differentiability of the objective of an ODE-constrained problem 
    w.r.t.\ the switching times is investigated. The proof however illustrates that this mapping is in general 
    only directionally differentiable. 
    This issue has also been noticed in Appendix A in \cite{SOBG17}, but, despite this lack of differentiability, 
    algorithms from (smooth) nonlinear optimization are frequently applied to switching time optimization problems, 
    see, e.g., \cite{SOBG17} or the survey article \cite{ZA15} containing an overview over various algorithmic approaches 
    in this field. 
    We underline that the use of smooth nonlinear optimization algorithms for the 
    solution of nonsmooth problems is problematic, 
    as the famous example of the Wolfe-function demonstrates, 
    see, e.g., \cite[Section~1.3]{Alt04}.
    In \cite{DMG20}, a proximal gradient scheme is used
    to handle an $\ell_0$-penalty term on the switching points, 
    while the rest of the objective is assumed to be continuously differentiable w.r.t.\ the switching times. 
    The situation improves if one restricts the feasible set to switching patterns
    with strictly increasing switching times separated from the initial and final time, i.e., $0 < \tau_1 < \ldots < \tau_n < T$.
    In \cite{AH21, MBKK05}, it is shown that the objective and the solution operator of the state equation, respectively, 
    for certain types of ODE-constrained optimal control problems, are differentiable under this assumption.
    This approach
    has the drawback that topological changes in the switching pattern are not allowed, 
    since two switching points are forbidden to be equal.
    In \cite{FMOB13}, the authors observe that, even after discretization of the ODE, the objective is not differentiable w.r.t.\ 
    the switching points and for that reason resort to subgradient methods to solve the problem numerically. 
    The latter, however, only works for convex problems and can in general not be applied to \eqref{eq:ocp0}, which 
    is a non-convex problem, even if the state equation is linear and the objective is convex. Indeed, 
    our numerical experiments indicate that the switching-point-to-control operator is a substantial source of non-convexity.
    
    The PDE case is analyzed in \cite{RH16}, where differentiability of the reduced objective w.r.t.\ the 
    switching points on the feasible set $\PP$ is shown. The results show that the derivative can be continuously extended 
    from the set of strictly increasing switching points to the feasible set $\PP$. 
    Unfortunately, for the application of standard gradient-based methods, this does 
    not suffice, since these methods 
    in general require differentiability on an open superset of the feasible set. 
    This is even more true for non-feasible methods such as proximal gradient methods with 
    Nesterov acceleration like FISTA or active set strategies.  
    We will therefore construct a mapping that coincides with the switching-point-to-control operator on the feasible set
    but is continuously differentiable on $\R^n$.
    Our method of proof substantially differs from the one in \cite{RH16}. 
    While we treat the state equation on the whole time horizon in
    weak form, the authors from \cite{RH16} consider the equation piecewise in the intervals defined by the 
    switching times.
    With our differentiability results at hand, we are able to apply the known convergence theory for 
    the proximal gradient method for non-convex problems to our switching-point optimization problem.
    
    An alternative approach to optimization problems of the form \eqref{eq:ocp0} is to reinterpret them as 
    optimal control problems in $\BV([0,T];\{0,1\})$, the space of functions of bounded variation with values in $\{0,1\}$. 
    This approach has been pursued by various authors, we only mention \cite{LM22} and \cite{MW23}, where 
    the equivalence to a switching point optimization problem is used to establish optimality conditions. 
    This reinterpretation leads to a bunch of alternative algorithmic approaches such as 
    tailored trust region methods \cite{LM22}, Bellmann's dynamic programming approaches \cite{MW23}, 
    rounding strategies, see \cite{SJK11, HS13, MK20}
    and the references therein, and convex relaxations in combination with
    branch-and-bound \cite{BGM24a, BGM24b, BGM25}, 
    just to mention a few methods.\todo{hier jetzt noch den Bezug unserer Arbeit zu 
    Abschnitt~\ref{sec:global} erwaehnen}
    	
    	The plan of this paper is as follows:
    	In the next section, we introduce a mathematically rigorous definition of the switching point optimization problem 
    	including a  rigorous notion of solutions to the state equation based on the concept of maximal parabolic regularity. 
    	As we will see by means of elementary counterexamples, the switching-point-to-control operator is not differentiable 
    	in Bochner-Lebesgue spaces, but in Section~\ref{sec:diffbarkeit}, it is shown to be continuously Fr\'echet-differentiable 
    	in the dual of certain H\"older-in-time spaces. Together with a weak solution concept for the state equation, 
    	this allows us to prove the differentiability of the switching-point-to-state operator considered as mapping with values 
    	in $L^p(0,T)$, $p \in [1, \infty)$. Based on this, the derivation of first-order necessary optimality conditions in form 
    	of Karush-Kuhn-Tucker (KKT) conditions is straightforward. 
    	Section~\ref{sec:prox} is devoted to a proximal gradient method in form of a projected gradient method for the 
    	solution of \eqref{eq:ocp0}. 
    	Our theoretical findings from Section~\ref{sec:diffbarkeit} allow us to apply convergence results
    	from the literature to our proximal point algorithm. 
    We emphasize that it is our aim to use a method as robust as possible that allows for a rigorous 
    convergence analysis in function space rather than to develop a method with maximum efficiency. 
    Section~\ref{sec:numerics} reports the results of four numerical tests with known exact solution. 
    As we will see, in agreement with the theoretical predictions, the proximal gradient method converges to 
    stationary point, but is rarely able to detect global minimizers due to the non-convexity caused 
    by the switching-point-to-control operator.
    In preparation of a branch-and-bound algorithm for the global optimization of \eqref{eq:ocp0}, 
    we therefore investigate the convex hull of all pairs of feasible switching vectors~$\tau$ 
    and their correspoding switching functions
    and provide a precise characterization of the latter in terms of linear (in-)equalities, see Section~\ref{sec:convex}.
    The paper ends with an appendix containing auxiliary results and the ODE-counterexample already mentioned above.
	
	\section{Notation, Standing Assumptions, and Rigorous Mathematical Formulation of the Optimization Problem}
    \label{sec:assus}
	
	Given a metric space $\XX$, we denote the open ball in $\XX$ around $0$ with radius $\rho > 0$
	by $B_\XX(0,\rho)$.
	If $\XX$ and $\YY$ are two linear normed spaces, the space of linear and bounded operators 
	from $\XX$ to $\YY$ is denoted by $L(\XX,\YY)$. The dual of $\XX$ is denoted by $\XX^*$ and 
	for the dual pairing, we write $\dual{\cdot}{\cdot}_{\XX^*, \XX}$, where we frequently neglect the subscript 
	if there is no risk of ambiguity. By 
	\begin{equation*}
        \BV(0,T) := \{ v \in L^1(0,T) : \, D v \in \frakM(0,T) \},
    \end{equation*}
	we denote the space of functions of bounded variation. Herein, $D v$ denotes the distributional derivative of $v$ 
	and $\mathfrak{M}(0,T)$ is the space of regular Borel measures on $(0,T)$, which is the dual of $C_0(0,T)$, 
	the space of continuous functions vanishing in $t = 0$ and $t = T$.	Equipped with the norm
	\begin{equation*}
	    \|v\|_{\BV(0,T)} := \|v\|_{L^1(0,T)} + |v|_{\TV}
	    \quad \text{with}\quad 
	    |v|_{\TV} := \|Dv\|_{\frakM(0,T)},
	\end{equation*}
	$\BV(0,T)$ becomes a Banach space.
	
	\begin{assumption}[Spaces]
		\begin{enumerate}[label=\textup{(\roman*)}]
			\item $T > 0$ is a fixed final time.
			\item $X$ and $\DD$ are reflexive Banach spaces with $\DD \embed^d X$.
			\item $A: X \to X$ is a linear, unbounded, and closed operator with $\DD$ being its domain of definition, i.e., 
			$\DD = \{x\in X \colon \|Ax\|_X < \infty\}$.
			We moreover assume that zero is not contained in the spectrum of $A$ so that $\|A \cdot \|_X$ is an 
			equivalent norm on $\DD$.
		\end{enumerate}
		\label{assu:spaces}
	\end{assumption}
	
    Given the spaces $X$ and $\DD$, numbers $a, b\in \R$, $a < b$, and an integrability exponent $q \in [1, \infty)$,
    we introduce the parabolic solution space as usual by	
	\begin{equation*}
		\W^{q}(a,b; \DD, X) 
		:= W^{1,q}(a,b;X) \cap L^q(a,b;\DD).
	\end{equation*}
    Herein, $L^q(a,b;\DD)$ denotes the Bochner-Lebesgue space of all Bochner-measurable 
    abstract functions with values in $\DD$ that are Bochner integrable to the power $p$, 
    while $W^{1,q}(a,b;X)$	is the Bochner-Sobolev space, i.e., all functions in $L^q(a,b;X)$ whose 
    distributional time derivative is an element of $L^q(a,b;X)$, too.
    If $(a,b) = (0,T)$, we simply write $\W^{q}(\DD, X) := \W^{q}(0,T; \DD, X)$.
	Given $\theta \in (0,1)$ and $r\in [1, \infty]$, we denote the real interpolation space between $X$ and $\DD$ by 
	\begin{equation*}
		Y_{\theta, r} := (X, \DD)_{\theta, r}. 
	\end{equation*}
	If $r\in (1, \infty)$ and $\theta = 1/r'$, where $r' = r/(r-1)$ denotes the conjugate exponent, we write 
	\begin{equation*}
		Y_r := (X, \DD)_{1/r', r}.
	\end{equation*}
	It is known that $\W^q(\DD,X)$ continuously embeds into $C([0,T];Y_q)$, see, e.g., \cite[Theorem~III.4.10.2]{Ama95}.
	Consequently, there is a linear and continuous trace operator $\gamma_0: \W^q(\DD,X) \to Y_q$ 
	such that $\gamma_0(v) = v(0)$ in $Y_q$ for all $v\in \W^q(\DD,X)$.
	The subspace of $\W^q(\DD,X)$ of functions with vanishing trace is denoted by 
	\begin{equation}\label{eq:W0}
		\W^q_0(\DD,X) := \{ v\in \W^q(\DD,X) \colon \gamma_0 v = 0\}.
	\end{equation}
		
	\begin{assumption}[Nonlinearity]\label{assu:f}
		Here and in the following, we will consider the nonlinearity $f$ in different domains and ranges. To ease 
		notation, it will always be denoted by the same symbol.
        We assume that there exists an exponent $q \in (1, \infty)$ such that the following conditions are fulfilled:			
		\begin{enumerate}[label=\textup{(\roman*)}]
			\item\label{it:f1}
			The nonlinearity $f$ is continuous from $\DD$ to $Y_q = (X, \DD)_{1/q', q}$.
			Moreover, it satisfies the growth condition
			\begin{equation*}
				\|f(v)\|_{Y_q} \leq c_1 + c_2 \, \|v\|_{\DD}^{q} \quad \forall \, v\in \DD
			\end{equation*} 
			with constants $c_1 \in \R$ and $c_2 \geq 0$.   
			
			
			\item\label{it:f3} $f$ is continuously Fr\'echet-differentiable from $\DD$ to $Y_q$.
			Moreover, the Fr\'echet-derivative satisfies the growth condition
			\begin{equation*}
				\|f'(v)\|_{L(\DD;Y_q)} \leq d_1 + d_2 \, \|v\|_{\DD}^{q/r} \quad \forall \, v\in \DD
			\end{equation*}
			with an exponent satisfying $ r > q'$ and constants $d_1 \in \R$ and $d_2 \geq 0$. 
			%
		\end{enumerate}
	\end{assumption}
	
	\begin{lemma}\label{lem:fFrechet}
	    Under Assumption~\ref{assu:f}, 
		the Nemyzki operator associated to $f$, for the ease of notation denoted by the same symbol, 
		continuously maps $L^{q}(0,T;\DD)$ to $L^1(0,T;Y_q)$. 
		Furthermore, it is continuously Fr\'echet-differentiable from $L^{q}(0,T;\DD)$ to $L^1(0,T;Y_q)$ 
		with Fr\'echet-derivative
		\begin{equation*}
			(f'(v)h)(\cdot ) = f'(v(\cdot)) h(\cdot), \quad v, h\in L^q(0,T;\DD).
		\end{equation*}
	\end{lemma}
	
	\begin{proof}
		The first assertion follows from \cite[Theorem~1 and 4]{GKT92}. The second one 
		is a consequence of \cite[Theorem~7]{GKT92}.
	\end{proof}
	
	Using again the embedding results of \cite[Theorem~III.4.10.2]{Ama95}, we find
	\begin{equation}\label{eq:maxparregembeddual}
		\W^{q'}(X^*, \DD^*) \embed C([0,T]; (\DD^*, X^*)_{1/q, q'}) 
	\end{equation}
	and, due to the duality of real interpolation functors, there holds
	\begin{equation}\label{eq:functordual}
		Y_q^* = [(X, \DD)_{1/q', q}]^* = (X^*, \DD^*)_{1/q', q'} = (\DD^*, X^*)_{1/q, q'} .
	\end{equation}
	Therefore, the trace operator at final time $\gamma_T:  \W^{q'}(X^*,\DD^*) \to Y_q^*$ is well defined 
	and, similarly to \eqref{eq:W0}, we introduce the following space for functions with vanishing trace at final time:
	\begin{equation*}
		\W_T^{q'}(X^*, \DD^*) 
		:= \{v\in W^{1,q'}(0,T;\DD^*) \cap L^{q'}(0,T;X^*) \colon \gamma_T v = 0\}.
	\end{equation*}
	In all what follows, we will simply write $\dual{\cdot}{\cdot}$ for the dual pairing in $\W_T^{q'}(X^*,\DD^*)$
	and omit the index to ease notation.
	Moreover, given a function $u \in L^q(0,T;X)$, the mapping 
	\begin{equation}\label{eq:embedLqWstar}
		\W_T^{q'}(X^*,\DD^*) \ni v \mapsto \int_0^T \dual{u(t)}{v(t)}_{X, X^*}  \,\d t \in \R
	\end{equation}
	defines an element of $[\W_T^{q'}(X^*,\DD^*)]^*$, which, with a slight abuse of notation, will be also denoted by $u$. 
	
	\begin{assumption}[Initial state]
		The initial state $y_0$ is fixed throughout the paper and satisfies $y_0\in Y_q$.
	\end{assumption}

	\begin{definition}
		Consider the equation 
		\begin{equation}\label{eq:stateeqg}
			y' + A y + f(y) = u,\quad y(0) = y_0.
		\end{equation}
		\begin{enumerate}[label=\textup{(\roman*)}]
			\item Let $u\in L^q(0,T;X)$ be given.
			A function $y$ is called \emph{strong solution} of \eqref{eq:stateeqg}, 
			if it satisfies $y\in \W_0^q(\DD, X)$ and fulfills the equation a.e.\ in time, i.e.,
			\begin{alignat}{3}
				y'(t) + A y(t) + f(y(t)) &= u(t) & \quad & \text{in } X \quad \text{f.a.a.\ } t\in (0,T), \label{eq:strong1}\\
				y(0) &= y_0 & & \text{in } Y_q. \label{eq:strong2}
			\end{alignat}
			\item Let $u\in [\W_T^{q'}(X^*, \DD^*)]^*$ be given.
			A function $y$ is called \emph{weak solution} of \eqref{eq:stateeqg}, 
			if it satisfies $y \in L^q(0,T;\DD)$ and fulfills the state equation in the following weak sense:
			\begin{equation}\label{eq:weak}
				\begin{aligned}
					\int_0^T \big( & - \dual{y(t)}{v'(t)}_{\DD,\DD^*}  + \dual{y(t)}{A^*v(t)}_{\DD,\DD^*} 
					+ \dual{f(y(t))}{v(t)}_{Y_q, Y_q^*} \big)\,\d t   \\[-1ex]
					& = \dual{u}{v} + \dual{y_0}{\gamma_0 v}_{Y_q, Y_q^*}
					\quad \forall\, v\in \W_T^{q'}(X^*,\DD^*) .
				\end{aligned}
			\end{equation}
		\end{enumerate}
		We will employ the notions of weak and strong solutions also for other parabolic equations involving $A$, 
		in particular linear equations. Since the adaptation of this definition is straight forward, we do not go into detail.
	\end{definition}
	
	Some words concerning the nonlinearity in the definition of a weak solution are in order. 
	First note that, by \eqref{eq:maxparregembeddual} and \eqref{eq:functordual}, 
	there holds that $\W_T^{q'}(X^*,\DD^*) \embed L^\infty(0,T;Y_q^*)$. 
	In addition, since $X$ and $\DD$ are reflexive, $Y_q$ satisfies the Radon-Nikod\'ym condition
	and thus we have $[L^1(0,T;Y_q)]^* = L^{\infty}(0,T;Y_q^*)$. 
	Therefore, since $f$ maps $L^q(0,T;\DD)$ into $L^1(0,T;Y_q)$
	according to Lemma~\ref{lem:fFrechet}, the integral 
	$\int_0^T \dual{f(y)}{v}_{Y_q, Y_q^*}\,\d t$ is well defined for every $y\in L^q(0,T;\DD)$ and $v\in \W_T^{q'}(X^*,\DD^*)$.
	
	\begin{definition}
		Let $q\in [1, \infty)$ be given.
		A linear and bounded operator $A : \DD\to X$ has \emph{maximal parabolic $L^q$-regularity} in these spaces, if,
		for every $g\in L^q(0,T;X)$, the equation
		\begin{equation*}
			y' + A y = g, \quad y(0) = 0
		\end{equation*}
		admits a unique strong solution in $\W_0^q(\DD, X)$, i.e., the operator 
		$\partial_t + A : \W_0^q(\DD, X) \to L^q(0,T;X)$ is continuously invertible.
	\end{definition}
	
	The next lemma collects some useful facts on maximal parabolic regularity.
	For its proof and further details, we refer to \cite[Section~2 and 3]{Ama03}.
    The claim in \ref{it:maxparreg3} is proven in \cite[Lemma~42]{HMS17}.
    
	\begin{lemma}\label{lem:maxparreg}
		\begin{enumerate}[label=\textup{(\roman*)}]
			\item A linear and bounded operator $A$ has maximal parabolic $L^q$-regularity if and only if
			the operator $(\partial_t + A, \gamma_0): \W^q(\DD,X) \to L^q(0,T;X) \times Y_q$ 
			is continuously invertible.
			\item If an operator $A$ has maximal parabolic $L^q$-regularity for some $q\in [1, \infty)$, 
			then it has maximal parabolic $L^s$-regularity for all $s\in [1, \infty)$.
			\item\label{it:maxparreg3}
			If an operator $A: \DD\to X$ has maximal parabolic $L^q$-regularity, then its adjoint 
			$A^* : X^* \to \DD^*$ admits maximal parabolic $L^{q'}$-regularity such that the operator 
			$(-\partial_t + A^*, \gamma_T): \W^{q'}(X^*,\DD^*) \to L^{q'}(0,T;\DD^*) \times Y_q^*$ is continuously invertible.
		\end{enumerate}
	\end{lemma}
		
	\begin{assumption}[Maximal Parabolic Regularity]\label{assu:maxparreg}
		Throughout the paper, the linear and bounded operator 
		$A: \DD \to X$ is assumed to admit maximal parabolic $L^q$-regularity.
	\end{assumption}
	
	\begin{lemma}
	  Every strong solution of \eqref{eq:stateeqg} is a weak solution. 
	 	
		The reverse statement is only true under additional assumptions on $u$ and $f$, for instance:
		If $u\in L^q(0,T;X)$ and $f$ satisfies the following
		\begin{alignat}{3}
			& \exists\, s\in (1,q], \, \tilde c_1\in \R, \, \tilde c_2 \geq 0 \, \colon 
			& \quad & \|f(v)\|_{X} \leq c_1 + c_2 \, \|v\|_{\DD}^{q/s} \quad \forall \, v\in \DD, \label{eq:fgrowth} \\
			& \forall\, M > 0 \; \exists\, k = k(M) \geq 0 \, \colon 
			& & \|f(v)\|_X \leq k \quad \forall \, v \in Y_r \text{ with } \|v\|_{Y_r} \leq M \label{eq:fbound}
		\end{alignat} 
		with $r = q(1 + \frac{1}{q} - \frac{1}{s})$, then every weak solution is also a strong solution.
	\end{lemma}
	
	\begin{proof}
		Let $y\in \W^q(\DD,X)$ be a strong solution associated with $u\in L^q(0,T;X)$.  Then, it follows from 
		\eqref{eq:strong1} that $f(y) \in L^q(0,T;X)$.
		Hence, we are allowed to multiply \eqref{eq:strong1} with an arbitrary test function 
		$v\in \W_T^{q'}(X^*,\DD^*) \embed L^{q'}(0,T;X^*)$ and to integrate the resulting equation.
		By applying the formula of integration by parts from \cite[Prop.~5.1]{Ama03}, we obtain 
		\begin{equation*}
			\begin{aligned}
				0 &= 
				\begin{aligned}[t]
					\int_0^T \big( \dual{y'(t)}{v(t)}_{X, X^*} & + \dual{A y(t)}{v(t)}_{X, X^*} \\
					& + \dual{f(y(t))}{v(t)}_{X, X^*} - \dual{u(t)}{v(t)}_{X, X^*} \big)\,\d t         
				\end{aligned} \\
				&= 
				\begin{aligned}[t]
					\int_0^T \big( -\dual{y(t)}{v'(t)}_{\DD, \DD^*} & + \dual{y(t)}{A^*v(t)}_{\DD, \DD^*} 
					+ \dual{f(y(t))}{v(t)}_{Y_q, Y_q^*} \big)\,\d t \\
					& - \dual{\bar u}{v}
					- \dual{y_0}{\gamma_0 v}_{Y_q, Y_q^*}
				\end{aligned} 
			\end{aligned}
		\end{equation*}
		such that $y$ is indeed a weak solution, too.
		
		To show the reverse assertion under the additional assumptions, first observe that the growth condition in 
		\eqref{eq:fgrowth} implies that $f$ maps $L^q(0,T;\DD)$ to $L^s(0,T;X)$, see \cite[Theorem~1]{GKT92}.
		Now, if $y \in L^q(0,T;\DD)$ is a weak solution of \eqref{eq:stateeqg} associated with 
		$u\in L^q(0,T;X)$, then by testing 
		\eqref{eq:weak} with $v = \varphi \, w$ with arbitrary $\varphi\in  C^\infty_c(0,T)$ and arbitrary $w\in X^*$, we find
		that 
		\begin{equation}\label{eq:distrderiv}
			- \int_0^T \varphi'(t) \, y(t)\,\d t
			= \int_0^T \big(u(t) - Ay(t) - f(y(t))\big)\varphi(t)\,\d t \quad \text{in } X.
		\end{equation}
		Hence, the distributional time derivative of $y$ is a regular distribution in $L^s(0,T;X)$, i.e., 
		\begin{equation*}
			y \in W^{1,s}(0,T;X) \cap L^q(0,T;\DD) \embed C([0,T]; (X,\DD)_{1/r', r}) = C([0,T];Y_r)
		\end{equation*}
		with $r = q(1 + \frac{1}{q} - \frac{1}{s})$, cf.\ \cite[Lemma~1.4.4]{Mei17}. 
		Thus, \eqref{eq:fbound} implies $f(y) \in L^\infty(0,T;X)$ and therefore, by returning again to \eqref{eq:distrderiv}, 
		we find that 
		\begin{equation}\label{eq:weakstrong}
			y' = u - Ay - f(y) \in L^q(0,T;X).
		\end{equation}
		This entails that \eqref{eq:strong1} is fulfilled and that $y\in \W^q(\DD,X)$, i.e., $y$ has the regularity of a strong solution. 
		It remains to verify the initial condition. To this end, we again apply the formula of integration by parts from 
		\cite[Prop.~5.1]{Ama03} to \eqref{eq:weak} and use \eqref{eq:weakstrong} in order to obtain
		\begin{equation*}
			\dual{\gamma_0 y}{\gamma_0 v}_{Y_q, Y_q^*}
			= \dual{y_0}{\gamma_0 v}_{Y_q, Y_q^*} \quad \forall\, v\in \W_T^{q'}(X^*,\DD^*) .
		\end{equation*}
		Since $A$ and thus also $A^*$ satisfies maximal parabolic regularity by assumption, the trace is surjective, 
		i.e., $\gamma_0(\W_T^{q'}(X^*, \DD^*)) = Y_q^*$ and therefore, $\gamma_0 y = y_0$ in $Y_q$.
	\end{proof}
	
	Next we turn to the form functions $\psi_i$ in \eqref{eq:stateeq}. 
	
	\begin{assumption}[Form functions]\label{assu:formfunc}
		For all $i = 1, \ldots, n-1$, the form function $\psi_i$ satisfies 
		\begin{equation*}
			\psi_i \in Y_{\vartheta, \infty} = (X, \DD)_{\vartheta, \infty}
			\quad \text{with} \quad \vartheta \in \big(\tfrac{1}{q'},1\big).
		\end{equation*}        
	\end{assumption}
	
	
	Given a tuple of switching times $0 \leq \tau_1 \leq \tau_2 \leq \ldots \leq \tau_n \leq T$, 
	the switching-time-to-control mapping is given by 
	\begin{equation}\label{eq:Bnaive}
		B : \R^n \ni \tau \mapsto 
		B(\tau) := \sum_{i=1}^{n-1}\chi_{[\tau_i, \tau_{i+1})} \,\psi_i \in L^\infty(0,T;(X, \DD)_{\vartheta, \infty}).
	\end{equation}
	Equipped with this range space, however, the mapping $B$ turns out to be non-differentiable, 
	see Section~\ref{sec:diffB} below. 
	For this reason, we consider $B$ as a mapping in a weaker space. To be more precise, 
	we aim to consider $B$ as a mapping in the dual of $\W^{q'}_T(X^*, \DD^*)$, which is just the test space 
	in the definition of a weak solution, see \eqref{eq:weak}.
	According to \cite[Theorem~3]{Ama01}, for all $\alpha \in [0,  \vartheta - 1/q')$, this space satisfies 
	\begin{equation}\label{eq:defZ}
		\W^{q'}_T(X^*, \DD^*) \embed C^{0,\alpha}(0,T;Z)
		\quad \text{with} \quad 
		Z := (X^*,\DD^*)_{1 - \vartheta,1}.
	\end{equation}
	Note that $\vartheta > 1/q'$ by Assumption~\ref{assu:formfunc} so that the interval for $\alpha$ 
	is nonempty. Now the duality of real interpolation functors implies
	\begin{equation*}
		Z^* = [(X^*,\DD^*)_{1 - \vartheta,1}]^* = (X, \DD)_{1 - \vartheta, \infty} = (\DD, X)_{\vartheta, \infty} = Y_{\vartheta, \infty}
	\end{equation*}
	and therefore, we can consider $B$ with range in  $[C^{0,\alpha}(0,T;Z)]^*$ and $[\W^{q'}_T(X^*, \DD^*) ]^*$, 
	respectively, by setting
	\begin{equation}\label{eq:Bweak}
		\begin{aligned}
			& B: \R^n \to [C^{0,\alpha}(0,T;Z)]^* \embed [\W^{q'}_T(X^*, \DD^*) ]^*, \\
			& \dual{B(\tau)}{v}_{[C^{0,\alpha}(0,T;Z)]^*, C^{0,\alpha}(0,T;Z)}
			:= \int_0^T \sum_{i=1}^{n-1}\chi_{[\tau_i, \tau_{i+1})}(t) \, \dual{\psi_i}{v(t)}_{Z^*, Z} \,\d t .
		\end{aligned}
	\end{equation}
	
	Now we have everything at hand to rigorously define the optimization problem under consideration:
	\begin{equation}\tag{OCP}\label{eq:ocp}
		\left\{\quad
		\begin{aligned}
			\min \quad & J(y,\tau) \\
			\text{s.t.} \quad & \text{$y\in L^q(0,T;\DD)$ is a weak solution associated with $u = B(\tau)$},\\
			& \tau \in \PP, 
		\end{aligned}
		\right.
	\end{equation}    
	where $J: L^q(0,T;\DD) \times \R^n \to \R$ is a given objective, for which we require the following
	
	\begin{assumption}[Objective]\label{assu:obj}
		The objective functional $J$ is assumed to be Fr\'echet-differentiable from 
		$L^q(0,T;\DD) \times \R^n$ to $\R$.
	\end{assumption}
	
	We finally emphasize that all of the above assumptions are standing assumptions that are 
	tacitly assumed to hold throughout the entire paper. 
	
	\section{Sensitivity Analysis}\label{sec:diffbarkeit}
	
	\subsection{Differentiability of the Switching-Point-to-Control Operator}\label{sec:diffB}
	
	In this subsection, we prove the Fréchet-differentiability of the switching-point-to-control operator $B$.
	As indicated above, as a mapping with range in a Lebesgue-Bochner-space, $B$ is not differentiable, 
	as the following example demonstrates.
	
	\begin{example}\label{ex:Bnotsmooth}
		Let $X = \DD = \R$, $n=2$, and $\psi_1 = 1$. Suppose
		\begin{align*}
			B: \R^{2} \to  L^p(0,T), \; \tau \mapsto \chi_{[\tau_{1}, \tau_{2})}, 
		\end{align*}
		was (G\^ateaux)-differentiable for some $\tau = (\tau_1, \tau_2) \in \R^2$ and $ 1 \leq p < \infty$. 
		Then, for every $h \in \R^2$, we would have 
		\begin{align*}
			\frac{B(\tau + \rho h) - B(\tau)}{\rho}  \xrightarrow{L^p} B'(\tau)h 
		\end{align*}
		for $\rho \searrow 0$. It would follow that there is a null sequence $(\rho_n)_{n \in \N} \subset (0, \infty)$ such that 
		\begin{align*}
			\frac{B(\tau + \rho_n h)(t) - B(\tau)(t)}{\rho_n} \xrightarrow{n \to \infty} B'(\tau)h(t)
			\quad \text{f.a.a.\ } t \in (0,T).
		\end{align*}
		However, for the pointwise limit, we obtain
		\begin{align*}
			\frac{B(\tau + \rho_n h)(t) - B(\tau)(t)}{\rho_n} 
			= \frac{1}{\rho_n} \left(\chi_{[\tau_1 + \rho_n h_1, \tau_2 + \rho_n h_2)}(t)  
			- \chi_{[\tau_1, \tau_2)}(t) \right) \xrightarrow{n \to \infty} 0
		\end{align*}
		for almost all $t \in (0,T)$, yielding $B'(\tau)h = 0$ and therefore $B'(\tau) = 0$. 
		Hence, $B$ must be constant on every domain where it is G\^ateaux-differentiable.
		This is obviously not the case, if $\tau_1 < \tau_2$, $\tau_1 < T$, and/or $\tau_2 > 0$.
	\end{example}
	
	We therefore consider the operator defined in \eqref{eq:Bweak} instead, which fits to the notion of a weak solution 
	as it maps into $[\W^{q'}_T(X^*, \DD^*) ]^*$. However, the mapping is still not differentiable.
	Consider, for instance, the same setting as in Example~\ref{ex:Bnotsmooth}, i.e.,
	$X = \DD = \R$ and $n =2$, and set $\tau_1 = \tau_2$. Then, as a consequence of 
	$\chi_{[\tau_1, \tau_2 - \varepsilon)}(t) = 0$ for all $\varepsilon > 0$, 
	  the left-sided derivative vanishes in $\tau =(\tau_1, \tau_2)$, while we obtain for the right-sided derivative 
	\begin{equation*}
	    \frac{1}{\varepsilon} \int_0^T \big( \chi_{[\tau_1, \tau_2 + \varepsilon)}(t) - \chi_{[\tau_1, \tau_2)}(t)\big) 
	    v(t) \,\d t  \xrightarrow{\varepsilon \searrow 0} v(\tau_1) = \delta_{\tau_1}(v)
	\end{equation*}
	for all $v\in C([0,T])$.
	To resolve this issue,  we modify the definition of the characteristic function by setting
	\begin{equation}\label{eq:defbarchi}
		\bar \chi_{[a,b)} (t) :=
		\begin{cases}
			1, & \text{for } a \leq t  < b\\
			-1, & \text{for } b < t \leq a, \\ 
			0, & \text{else,}
		\end{cases}
	\end{equation}
	for $a,b, t \in \R$. By means of a distinction of cases, one easily verifies that, for every $a,b, h_1, h_2 \in \R$, 
	it holds that 
	\begin{equation}\label{eq:barchieq}
		\bar\chi_{[a + h_1, b+ h_2)}(t) 
		= 	\bar\chi_{[b , b + h_2)} - 	\bar\chi_{[a,  a + h_1)} 
		+ \bar\chi_{[a,  b)} \quad \text{f.a.a.\ } t\in \R. 
	\end{equation}
	
	Given $v\in C^{0,\alpha}(0,T;Z)$, we define $V$ as its extension on $\R$ by constant continuation, i.e., 
	\begin{equation}\label{eq:extensiononR}
		V: \R \to Z,\quad 
		V(t) :=  v\big(\min\{\max\{t,0\}, T\}\big).
	\end{equation}
	Note that $V$ inherits its supremum-norm as well as its H\"older constant from $v$ such that 
	$\|V\|_{C^{0,\alpha}(\R;Z)} = \|v\|_{C^{0,\alpha}(0,T;Z)}$. 
	Given this extension and the modified characteristic function $\bar\chi$, we redefine the switching-time-to-control mapping by
	\begin{equation}\label{eq:Bweak2}
		\begin{aligned}
			& \bar B: \R^n \to [C^{0,\alpha}(0,T;Z)]^*, \\
			& \dual{\bar B(\tau)}{v}_{[C^{0,\alpha}(0,T;Z)]^*, C^{0,\alpha}(0,T;Z)}
			:= \int_{\R} \sum_{i=1}^{n-1}\bar\chi_{[\tau_i, \tau_{i+1})}(t) \, \dual{\psi_i}{V(t)}_{Z^*, Z} \,\d t .
		\end{aligned}
	\end{equation}	
	By the very definition of $\bar\chi_{[\tau_i, \tau_{i+1})}$, it directly follows that $\bar B$ equals $B$ on the 
	set of feasible switching points. For later reference, we formulate this finding in the following
	
	\begin{lemma}\label{lem:BgleichbarB}
		For all $\tau \in \PP$ and all $i = 1, \ldots,n$, 
		we have $\chi_{[\tau_i, \tau_{i+1})} = \bar\chi_{[\tau_i, \tau_{i+1})}$ 
		and thus $B|_{\PP} = \bar B|_{\PP}$.
	\end{lemma}
	
	\begin{theorem}\label{thm:barBdiff}
		The switching-point-to-control operator, as defined in \eqref{eq:Bweak2}, is continously Fréchet-differentiable 
		from $\R^n$ to $[C^{0,\alpha}(0,T;Z)]^*$. Its directional derivative at $\tau\in \R^n$ in direction $h\in \R^n$
		is given by
		\begin{equation}\label{eq:barBabl}
			\dual{\bar B'(\tau)h}{v}_{[C^{0,\alpha}(0,T;Z)]^*, C^{0,\alpha}(0,T;Z)}  = 
			\sum_{i=1}^{n-1}  \langle \psi_i, h_{i+1} \, V(\tau_{i+1})  - h_i \, V(\tau_{i}) \rangle_{Z^*, Z} . 
		\end{equation}
	\end{theorem}
	
	\begin{proof}
		Throughout the proof, let us abbreviate $\ZZ := C^{0,\alpha}(0,T;Z)$ to ease notation.
		Using \eqref{eq:barchieq}, we obtain for arbitrary $\tau, h \in \R^n$ that 
		\begin{align*}
			& \| \bar B(\tau + h) - \bar B(\tau) - \bar B'(\tau)h \|_{\ZZ^*} \\ 
			& \quad = \sup_{\|v\|_{\ZZ} \leq 1} \,
			\begin{aligned}[t]
				\Big( & \sum_{i=1}^{n-1} \int_\R \big(\bar \chi_{[\tau_{i+1}, \tau_{i+1} + h_{i+1})} (t) - 
				\bar \chi_{[\tau_i, \tau_{i}+ h_i)} (t) \big) \dual{\psi_i}{V(t)}_{Z^*, Z} \,\d t \\ 
				& - \sum_{i=1}^{n-1} \langle \psi_i, h_{i+1} \, V(\tau_{i+1}) - h_i \, V(\tau_i) \rangle_{Z^*, Z} \Big)
			\end{aligned} \\
			& \quad \leq 
			\begin{aligned}[t]
				&\sum_{i=1}^{n-1} \sup_{\|v\|_\ZZ \leq 1} 
				\int_\R \bar \chi_{[\tau_{i+1}, \tau_{i+1} + h_{i+1})} (t) \langle \psi_i, V(t) \rangle_{Z^*, Z} \,\d t 
				- h_{i+1} \,\langle \psi_i, V(\tau_{i+1}) \rangle_{Z^*, Z} \\ 
				& + \sum_{i=1}^{n-1} \sup_{\|v\|_\ZZ \leq 1} 
				\int_\R  \bar \chi_{[\tau_{i}, \tau_{i} + h_{i})} (t) \langle \psi_i, V(t) \rangle_{Z^*, Z} \,\d t 
				- h_{i}\, \langle \psi_i, V(\tau_{i}) \rangle_{Z^*, Z} .
			\end{aligned}        		
		\end{align*}
		For every $i = 1, \ldots, n-1$ and every $k = 0,1$, we now have 
		\begin{align*}
			&  \sup_{\|v\|_\ZZ \leq 1} \int_\R  \bar \chi_{[\tau_{i+k}, \tau_{i+k} + h_{i+k})} (t) \langle \psi_i, V(t) \rangle_{Z^*, Z} \,\d t 
			- h_{i+k} \,\langle \psi_i, V(\tau_{i+k}) \rangle_{Z^*, Z} \\
			& \quad =  \sup_{\|v\|_\ZZ \leq 1} \int_{\tau_{i+k}}^{\tau_{i+k}+ h_{i+k}} 
			\big( \langle \psi_i, V(t) \rangle_{Z^*, Z}  - \langle \psi_i, V(\tau_{i+k}) \rangle_{Z^*, Z} \big)\,\d t \\
			& \quad \leq  \sup_{\|v\|_\ZZ \leq 1} \|\psi_i\|_{Z^*} \, \|V\|_{C^{0,\alpha}(\R;Z)} \,
			\int_{\tau_{i+k}}^{\tau_{i+k}+ |h_{i+k}|} (t - \tau_{i+k})^\alpha\,\d t \\
			& \quad = \sup_{\|v\|_\ZZ \leq 1} \|\psi_i\|_{Z^*} \,\|v\|_\ZZ\, \frac{1}{\alpha+1} \, |h_{i+k}|^{\alpha+1}
			= \frac{1}{\alpha+1} \, \|\psi_i\|_{Z^*}\, |h_{i+k}|^{\alpha+1}
		\end{align*}
		Thus, we arrive at
		\begin{equation*}
			\begin{aligned}
				\frac{\| \bar B(\tau + h) - \bar B(\tau) - \bar B'(\tau)h \|_{\ZZ^*}}{\|h\|_{\R^n}} 
				& \leq\frac{1}{\|h\|_{\R^n}} \,
				\frac{1}{\alpha+1} \sum_{i=1}^{n-1} \|\psi_i\|_{Z^*} \big( |h_{i+1}|^{\alpha+1} + |h_{i}|^{\alpha+1}\big) \\
				&  \leq \frac{2(n-1)}{\alpha+1}  \,\|h\|_{\R^n}^\alpha \max_{1\leq i\leq n-1} \|\psi_i\|_{Z^*} \to 0 \quad \text{as } h \to 0,
			\end{aligned}
		\end{equation*}
		which is the claimed Fr\'echet-differentiability.
		
		To show that the derivative is continuous, it suffices to show the continuity of 
		\begin{equation*}
			\delta^{(i)}: \R^n \to L(\R^n, \ZZ^*), \quad 
			\langle \delta^{(i)}(\tau)h, v \rangle_{\ZZ^*, \ZZ} := h_i \langle \psi, v(\tau_i))_{Z^*, Z}
		\end{equation*}        
		with arbitrary, but fixed $i \in \{1, \ldots, n\}$ and $\psi \in Z^*$. 
		For that, let $\{\tau^{(n)}\}_{n \in \N}$ be a sequence converging to some $\tau \in \R^n$. We then obtain 
		\begin{align*}
			& \| \delta^{(i)}(\tau^{(n)}) - \delta^{(i)}(\tau) \|_{L(\R^n, \ZZ^*)} \\
			&\qquad = \sup_{\| h \|_{\R^n} \leq 1} \sup_{\| v \|_{\ZZ} \leq 1} 
			\langle (\delta^{(i)}(\tau^{(n}) - \delta^{(i)}(\tau)) h, v \rangle_{\ZZ^*, \ZZ} \\
			&\qquad = \sup_{\| v \|_{\ZZ} \leq 1} \langle \psi, v(\tau^{(n)}_i) - v(\tau_i)\rangle_{Z^*, Z}  
			\leq \|\psi \|_{Z^*} \| \tau^{(n)}_i - \tau_i \|_{\R^n}^\alpha \to 0
		\end{align*}
		for $n \to \infty$. 
	\end{proof}

	\subsection{Differentiability of the Parabolic Solution Map in Weak Form}
	
	\begin{lemma}\label{lem:adjointeq}
		Let $\bar y \in L^q(0,T;\DD)$ be given.
		For every $h\in L^{q'}(0,T; \DD^*)$, the dual linearized equation 
		\begin{equation}\label{eq:adjeq0}
			- \varphi' + A^* \varphi + f'(\bar y)^* \varphi = h, \quad \varphi(T) = 0
		\end{equation}
		admits a unique strong solution in $\varphi\in \W_T^{q'}(X^*,\DD^*)$.
	\end{lemma}
	
	\begin{proof}
	    The arguments are based on a standard contraction argument, but, for convenience of the reader, we sketch the proof here. 
	    Let $t\in (0,T)$ be arbitrary.
		By Lemma~\ref{lem:maxparreg}\ref{it:maxparreg3},  
		for every $g\in L^{q'}(0,t;\DD^*)$ and every $\psi_t \in Y_q^* = (\DD^*, X^*)_{1/q, q'}$, 
		there exists a unique strong solution $\psi \in \W^{q'}(0,t;X^*,\DD^*)$ of 
		\begin{equation*}
			- \psi' + A^* \psi = g, \quad \psi(t) = \psi_t
		\end{equation*}
		and the associated solution operator 
		\begin{equation}\label{eq:defG}
			K_t := (-\partial_t + A^*, \gamma_t)^{-1} :  L^{q'}(0,t;\DD^*) \times Y_q^* \to \W^{q'}(0,t;X^*,\DD^*)
		\end{equation}
		is linear and continuous. Thanks to the embedding 
		$\W^{q'}(0,t;X^*,\DD^*) \embed C([0,t];Y_q^*)$ already used above, we can consider $K_t$ 
		as an operator with values in $C([0,t]; Y_q^*)$, which we denote by the same symbol for simplicity.
		The associated operator norm is abbreviated by $\|K_t\|$.
		From Assumption~\ref{assu:f}\ref{it:f3}, it follows that $f$ is Fr\'echet-differentiable from 
		$L^q(0,T;\DD)$ to $L^1(0,T;Y_q)$, see also Lemma~\ref{lem:fFrechet}, and the growth condition in    
		Assumption~\ref{assu:f}\ref{it:f3}, implies that 
		\begin{equation}\label{eq:fprimestargrowth}
			\|f'(\bar y)^*\|_{L(Y_q^*, \DD^*)} \leq d_1 + d_2 \, \|\bar y\|_\DD^{q/r} 
			\quad \text{a.e.\ in } (0,T)
		\end{equation}
		with $r > q'$.  As $\bar y \in L^q(0,T;\DD)$, it thus follows that $f'(\bar y)^* \in L^r(0,T;L(Y_q^*,\DD^*))$.
		Consequently, the fixed point mapping associated with \eqref{eq:adjeq0} given by
		\begin{equation*}
			\Phi_t: C([0,t]; Y_q^*) \ni \varphi \mapsto K_t(h - f'(\bar y)^*\varphi, \psi_t) \in C([0,t]; Y_q^*)
		\end{equation*}
		is well defined. To show that it is contractive, provided that $t$ is small enough, 
		let $\varphi_i \in C([0,t]; Y_q^*)$, $i =1 ,2$, be arbitrary. Then, H\"older's inequality yields
		\begin{equation}\label{eq:Gest}
			\begin{aligned}
				& \|\Phi_t(\varphi_1) - \Phi_t(\varphi_2)\|_{C([0,t]; Y_q^*)} \\
				& \qquad \leq \|K_t\| \, 
				\| f'(\bar y)^*\varphi_1 -  f'(\bar y)^*\varphi_2\|_{L^{q'}(0,t;\DD^*)}\\
				& \qquad \leq \|K_t\| 
				\Big(\int_0^t \|f'(\bar y(s))^*\|_{L(Y_q^*, \DD^*)}^{q'}\,\d s\Big)^{1/q'}\, \|\varphi_1 - \varphi_2\|_{C([0,t]; Y_q^*)} \\
				&\qquad \leq \|K_t\|
				\, \|f'(\bar y)^*\|_{L^{r}(0,T;L(Y_q^*,\DD^*))} \, 
				t^{\frac{r-q'}{r q'}} \, \|\varphi_1 - \varphi_2\|_{C([0,t]; Y_q^*)}.
			\end{aligned}
		\end{equation}
		Now, the norm of $K_t$ can be estimated independently of $t$, see Lemma~\ref{lem:normGt} in the appendix, 
		where an analogous result is proven for the adjoint of $K_t$. 
		Thus $\Phi_t$ is indeed a contraction for $t$ sufficiently small and, by Banach's contraction principle, 
		$\Phi_t$ admits a unique fixed point, which, by the mapping properties of $K_t$, is also an 
		element of $\W^{q'}(0,t;X^*, \DD^*)$. Hence, by construction of $K_t$, it is also a strong solution of \eqref{eq:adjeq}.
		A standard concatenation argument finally implies the existence of a unique strong solution on the whole 
		time interval.
	\end{proof}
	
	\begin{theorem}\label{thm:implfunc}
		Suppose that $\bar y\in L^q(0,T;\DD)$ is a weak solution of 
		\eqref{eq:stateeqg} associated with right hand side $\bar u \in [\W^{q'}_T(X^*, \DD^*)]^*$.
		Then there exist neighborhoods $\UU(\bar u) \subset [\W_T^{q'}(X^*,\DD^*)]^*$ of $\bar u$ 
		and $\YY(\bar y)\subset L^q(0,T;\DD)$ of $\bar y$
		and a continuously Fr\'echet-differentiable mapping $\bar S: \UU(\bar u) \to \YY(\bar y)$ 
		such that $\bar S(u)$, $u\in \UU(\bar u)$, is the only weak solution in $\YY(\bar y)$ of \eqref{eq:stateeqg}.     
		
		Furthermore, given $u\in \UU(\bar u)$ and an arbitrary direction $h\in [\W_T^{q'}(X^*,\DD^*)]^*$, 
		the derivative $\eta := \bar S'(u)h \in L^q(0,T;\DD)$ is given by the unique weak solution of 
		\begin{equation}\label{eq:lineqweak}
			\begin{aligned}
				\int_0^T \big(- \dual{\eta(t)}{v'(t)}_{\DD, \DD^*} 
				+ \dual{\eta(t)}{A^*v(t)}_{\DD, \DD^*} 
				+ \dual{f'(y(t))\eta(t)}{v(t)}_{Y_q, Y_q^*} \big)\,\d t  & \\[-1ex]
				= \dual{h}{v} \quad \forall\, v\in \W_T^{q'}(X^*,\DD^*)   & 
			\end{aligned}
		\end{equation} 
		where $y := \bar S(u)$.
	\end{theorem}
	
	\begin{proof}
		Let us define the following function associated with the notion of a weak solution
		\begin{equation*}
			\begin{aligned}
				& F: L^{q}(0,T;\DD) \times [\W_T^{q'}(X^*,\DD^*)]^* \to [\W_T^{q'}(X^*,\DD^*)]^*, \\
				& \dual{F(y,u)}{v} := 
				\begin{aligned}[t]
					& \int_0^T \big( \dual{y(t)}{v'(t)}_{\DD, \DD^*} + \dual{y(t)}{A^*v(t)}_{\DD, \DD^*} 
					+ \dual{f(y(t))}{v(t)}_{Y_q, Y_q^*} \big)\,\d t \\
					& \qquad - \dual{u}{v} + \dual{y_0}{\gamma_0 v}_{Y_q, Y_q^*}
				\end{aligned}
			\end{aligned}
		\end{equation*}
		so that $(\bar y, \bar u)$ is a root of $F$. 
		Moreover, since 
		\begin{equation*}
			L^q(0,T;\DD) \ni y \mapsto \int_0^T \dual{y}{(-\partial_T + A^*)(\cdot)} \,\d t \in [\W^{q'}_T(X^*, \DD^*)]^*
		\end{equation*}
		is linear and bounded and $f$ is continuously Fr\'echet-differentiable from $L^q(0,T;\DD)$ to $L^1(0,T;Y_q)$ 
		by Lemma~\ref{lem:fFrechet}, the mapping $F$ is partially Fr\'echet-differentiable w.r.t.\ $y$
		from $L^q(0,T;\DD)$ to $[\W^{q'}_T(X^*, \DD^*)]^*$ with continuous derivative given by
		\begin{equation*}
			\begin{aligned}
				& \dual{\partial_y F(\bar y, \tilde u)\eta}{v} \\
				&\qquad = \begin{aligned}[t]
					\int_0^T \big( & -\dual{\eta(t)}{v'(t)}_{\DD, \DD^*} \\[-1ex] 
					& + \dual{\eta(t)}{A^*v(t)}_{\DD, \DD^*} + \dual{f'(\bar y(t))\eta(t)}{v(t)}_{Y_q, Y_q^*} \big)\,\d t 
				\end{aligned}\\
				&\qquad = \dual{\eta(t)}{(- \partial_t + A^* + f'(\bar y)^*)v(t)}_{L^q(0,T;\DD), L^{q'}(0,T;\DD^*)}, 
				\quad v\in \W^{q'}_T(X^*, \DD^*).
			\end{aligned}
		\end{equation*}
		Note that $f'(\bar y)^* : [L^1(0,T;Y_q)]^* = L^\infty(0,T;Y_q^*) \to L^{q'}(0,T;\DD^*)$, 
		since $Y_q$ satisfies the Radon-Nikod\'ym condition, and again remember that 
		$\W^{q'}_T(X^*, \DD^*) \embed C([0,T];Y_q^*)$.
		We thus have
		\begin{equation*}
			\partial_y F(\bar y, \tilde u) = (- \partial_t + A^* + f'(\bar y)^*)^* : L^q(0,T;\DD) \to [\W^{q'}_T(X^*, \DD^*)]^* . 
		\end{equation*}
		By Lemma~\ref{lem:adjointeq}, we know that the operator 
		$- \partial_t + A^* + f'(\bar y)^* : \W^{q'}_T(X^*, \DD^*) \to L^{q'}(0,T;\DD^*)$ is continuously invertible and 
		so is its adjoint, which is just $\partial_y F(\bar y, \bar u)$. 
		Therefore, the implicit function theorem implies the assertion.
	\end{proof}
	
	\begin{corollary}\label{cor:implfunc}
		Assume that, for every $u\in L^q(0,T;X)$, there exists a unique weak solution of the state equation such that 
		there exists a control-to-state mapping $S: L^q(0,T;X) \to L^q(0,T;\DD)$. 
		Then, for every $u\in L^q(0,T;X)$, 
		there exists a open neighborhood $\UU(u)\subset [\W_T^{q'}(X^*, \DD^*)]^*$ such that $S$ can be extended 
		in $\UU(u)$ to an  operator $\bar S: \UU(u) \to L^q(0,T;\DD)$ 
		such that $\bar S |_{\UU(u) \cap L^q(0,T;X))} = S|_{\UU(u) \cap L^q(0,T;X)}$ 
		and $\bar S$ is continuously Fr\'echet-differentiable in $u$ with Fr\'echet-derivative in direction 
		$h \in [\W_T^{q'}(X^*, \DD^*)]^*$ given by the solution of \eqref{eq:lineqweak}.
		Moreover, $S$ itself is continuously Fr\'echet-differentiable and the Fr\'echet-derivative at $u \in L^q(0,T;X)$ in direction 
		$h \in L^q(0,T;X)$ also solves \eqref{eq:lineqweak}.
	\end{corollary}
	
	\begin{proof}
		Let $u\in L^q(0,T;X)$ be arbitrary. By Theorem~\ref{thm:implfunc}, there exists a neighborhood 
		$\UU(u) \subset [\W_T^{q'}(X^*, \DD^*)]^*$ around $u$ such that $S$ can be extended to a
		Fr\'echet-differentiable operator 
		$\bar S : \UU(u) \to L^q(0,T;\DD)$ and $\bar S(w)$ is a weak solution associated with $w \in \UU(u)$.
		Thus it holds that $S(w) = \bar S(w)$ for all $w \in \UU(u) \cap L^q(0,T;X)$,
		where, with a slight abuse of notation, we neglected the embedding operator associated with 
		$L^q(0,T;X) \embed [\W_T^{q'}(X^*, \DD^*)]^*$ as given in \eqref{eq:embedLqWstar}. 
		Therefore, we obtain for all $h\in L^q(0,T;X)$ with $u+ h \in \UU(u)$ that
		\begin{equation*}
			S(u + h) = \bar S(u+h) 
			= \bar S(u) + \bar S'(u)h + r(u;h) = S(u) + \bar S'(u)h + r(u;h),
		\end{equation*}
		where the remainder term satisfies
		\begin{equation*}
			\frac{\|r(u;h)\|_{L^q(0,T;\DD)}}{\|h\|_{L^q(0,T;X)}}
			\leq  C\, \frac{\|r(u;h)\|_{L^q(0,T;\DD)}}{\|h\|_{[\W_T^{q'}(X^*, \DD^*)]^*}} \to 0
			\quad \text{as } h \to 0 \text{ in } L^q(0,T;X),
		\end{equation*}  
		which shows that $S'(u) = \bar S'(u)$.
	\end{proof}
	
	\subsection{First-order Necessary Optimality Conditions}
	
	In view of Lemma~\ref{lem:BgleichbarB}, \eqref{eq:ocp} can be reformulated as follows:
	\begin{equation*}
		\eqref{eq:ocp} 
		\Longleftrightarrow
		\left\{\;
		\begin{aligned}
			\min \quad & J(y,\tau) \\
			\text{s.t.} \quad & \text{$y\in L^q(0,T;\DD)$ is a weak solution associated with $u = \bar B(\tau)$},\\
			& \tau \in \PP.
		\end{aligned}
		\right.
	\end{equation*}  
	Let $(\bar y, \bar \tau) \in L^q(0,T;\DD) \times \R^n$ be a locally optimal solution of \eqref{eq:ocp} and set 
	$\bar u = \bar B(\bar \tau)$. 
	According to its definition in \eqref{eq:Bweak2}, $\bar B$ is a mapping with range in $[C^{0,\alpha}(0,T;Z)]^*$
	with $Z$ as defined in \eqref{eq:defZ}.
	In the following, we will consider $\bar B$ with range in $[\W_T^{q'}(X^*, \DD^*)]^*$, which is feasible due to \eqref{eq:defZ}. 
	To ease notation, this operator is denoted by $\bar B$, too. Note that, due to 
	$[\W_T^{q'}(X^*, \DD^*)]^* \embed [C^{0,\alpha}(0,T;Z)]^*$, Theorem~\ref{thm:barBdiff} implies that 
	$\bar B$ is continuously Fr\'echet-differentiable from $\R^n$ to $[\W_T^{q'}(X^*, \DD^*)]^*$ 
	with derivative as given \eqref{eq:barBabl}.
	
	According to Theorem~\ref{thm:implfunc}, there exists a neighborhood $\UU(\bar u) \subset [\W_T^{q'}(X^*, \DD^*)]^*$ 
	and an operator $\bar S: \UU(\bar u) \to L^q(0,T;\DD)$
	such that $\bar S(u)$ is a weak solution associated with $u\in \UU(\bar u)$ and it is the only one in a neighborhood of $\bar y$. 
	Due to the continuity of $\bar B$, there exists $\varepsilon > 0$ such that 
	$\bar B(B_{\R^n}(\bar\tau, \varepsilon)) \subset \UU(\bar u)$.
	Let us consider the following optimization problem
	\begin{equation}\tag{OCP$_\varepsilon$}\label{eq:ocpeps}
		\left\{\quad
		\begin{aligned}
			\min \quad & J(y,\tau) \\
			\text{s.t.} \quad & y = \bar S(\bar B(\tau)),
			& \tau \in \PP \cap B_{\R^n}(\bar\tau, \varepsilon).
		\end{aligned}
		\right.
	\end{equation}
	Clearly, since $\bar S(u)$ is a weak solution, $\bar\tau$ is also a local minimizer of this problem. 
	Moreover, because $\bar B$ is differentiable from $\R^n$ to $[C^{0,\alpha}(0,T;Z)]^*$ by Theorem~\ref{thm:barBdiff}
	and $[\W_T^{q'}(X^*, \DD^*)]^* \embed [C^{0,\alpha}(0,T;Z)]^*$, it is also 
	Fr\'echet-differentiable from $\R^n$ to $[\W_T^{q'}(X^*, \DD^*)]^*$. 
	Furthermore, $\bar S$ is Fr\'echet-differentiable in $\UU(\bar u)$ by Theorem~\ref{thm:implfunc}
	and, thus the differentiability of $J$ by Assumption~\ref{assu:obj} along with the chain rule
	implies that $J \circ \bar S \circ \bar B: \R^n \to \R$ is Fr\'echet-differentiable in $\bar\tau$. 
	Thus $\bar\tau$ satisfies the following variational inequality
    \begin{equation}\label{eq:VI}
    \begin{aligned}
        \dual{\partial_y J(\bar y, \bar \tau)}{\bar S'(\bar u) \bar B'(\bar\tau)(\tau - \bar \tau)}_{L^{q'}(0,T;\DD^*), L^q(0,T;\DD)} & \\
		+ ( \partial_\tau J(\bar y, \bar\tau) , \tau - \bar\tau)_{\R^n} \geq 0 & 
		\quad \forall\, \tau \in \PP \cap B_{\R^n}(\bar\tau, \varepsilon)    
    \end{aligned}
    \end{equation}
	and hence, since $\bar\tau \in \interior(B_{\R^n}(\bar\tau, \varepsilon))$,
	\begin{equation*}
		\begin{aligned}
			& (\bar B'(\bar\tau)^* \bar S'(\bar u)^* \partial_y J(\bar y, \bar \tau) + \partial_\tau J(\bar y, \bar\tau), 
			\tau - \bar \tau)_{\R^n} \geq 0 \quad \forall\, \tau \in \PP \\
			& \qquad\qquad \Longleftrightarrow \quad
			\bar\tau = \Pi_{\PP}\big( \bar B'(\bar\tau)^* \bar S'(\bar u)^* \partial_y J(\bar y, \bar \tau) 
			+ \partial_\tau J(\bar y, \bar\tau)\big),
		\end{aligned}
	\end{equation*}
	where $\Pi_{\PP}: \R^n \to \R^n$ denotes the Eucildean projection on the polyhedron $\PP$.
	Let us define $p := \bar S'(\bar u)^* \partial_y J(\bar y, \bar \tau) \in \W^{q'}_T(X^*, \DD^*)$.
	As $\bar S(\bar u)^*$ is the solution operator of the weak form of the linearized equation in \eqref{eq:lineqweak}, 
	standard arguments based on the formula of integration by parts from \cite[Proposition~5.1]{Ama03} show that 
	$p$ is the unique strong solution of 
	\begin{equation}\label{eq:adjeq}
		- p' + A^* p + f'(\bar y)^* p = \partial_y J(\bar y, \bar \tau), \quad  
		p(T) = 0.
	\end{equation}
	Note that $\partial_y J(\bar y, \bar \tau) \in [L^q(0,T;\DD)]^* = L^{q'}(0,T;\DD^*)$ such that \eqref{eq:adjeq} 
	admits a unique solution according to Lemma~\ref{lem:adjointeq}.
	Indeed, testing \eqref{eq:adjeq} with $\eta = \bar S'(\bar u) h\in L^q(0,T;\DD)$ 
	with $h\in [\W_T^{q'}(X^*, \DD^*)]^*$ arbitrary yields
	\begin{equation*}
		\begin{aligned}
			& \dual{\bar S'(\bar u)^*\partial_y J(\bar y, \bar \tau)}{h} \\
			& \quad = \int_0^T \dual{\partial_y J(\bar y, \bar \tau)(t)}{\eta(t)}_{\DD^*, \DD}\,\d t \\
			&\quad = \int_0^T \big( -\dual{\eta(t)}{p'(t)}_{\DD, \DD^*}
			+ \dual{\eta(t)}{A^*p(t)}_{\DD, \DD^*} + \dual{\eta(t)}{f'(\bar y(t))^*p(t)}_{\DD, \DD^*} \big)\,\d t \\
			& \quad = \dual{h}{p}
		\end{aligned}
	\end{equation*}
	and, since $h$ was arbitrary, this gives the claim. Bearing in mind the structure of the derivative of $\bar B$ in 
	\eqref{eq:barBabl}, we thus obtain the following
	
	\begin{theorem}\label{thm:KKT}
		Let $(\bar y, \bar \tau) \in L^q(0,T;\DD) \times \R^n$ be a locally optimal solution of \eqref{eq:ocp}.
		Then there exists an adjoint state $p\in \W^{q'}_T(X^*, \DD^*)$ 
		such that the following optimality system is fulfilled:
		\begin{equation*}
			\begin{gathered}
				\begin{aligned}
					\int_0^T \big( - \dual{\bar y(t)}{v'(t)}_{\DD,\DD^*}  &+ \dual{\bar y(t)}{A^*v(t)}_{\DD,\DD^*} 
					+ \dual{f(\bar y(t))}{v(t)}_{Y_q, Y_q^*} \big)\,\d t   \\[-1ex]
					& = \dual{\bar B(\bar \tau)}{v} + \dual{y_0}{\gamma_0 v}_{Y_q, Y_q^*}
					\quad \forall\, v\in \W_T^{q'}(X^*,\DD^*) ,
				\end{aligned} \\[1ex]
				\begin{aligned}
					- p'(t) + A^* p(t) + f'(\bar y(t))^* p(t) &= \partial_y J(\bar y, \bar \tau)(t) 
					& & \text{in } \DD^* \text{ f.a.a.\ } t\in (0,T), \\
					p(T) &= 0 & & \text{in } Y_q^*,
				\end{aligned} \\[1ex]
				\begin{aligned}
					\bar \tau &= \Pi_{\PP}\big( \bar B'(\bar\tau)^* p + \partial_\tau J(\bar y, \bar\tau)\big), \\[1ex]
					\bar B'(\bar\tau)^* p &= 
					\begin{pmatrix}
						- \dual{\psi_1}{p(\tau_1)}_{Z^*, Z} \\
						\big(\dual{\psi_{i-1} - \psi_i}{p(\tau_i)}_{Z^*, Z}\big)_{i=2}^{n-1} \\
						\dual{\psi_{n-1}}{p(\tau_n)}_{Z^*, Z}
					\end{pmatrix}\,.
				\end{aligned}
			\end{gathered}
		\end{equation*}
	\end{theorem}

	\section{Proximal Gradient Method}\label{sec:prox}
	
	Using the differentiability results from Section~\ref{sec:diffbarkeit}, we can formulate and analyze 
	a proximal gradient method for the numerical solution of \eqref{eq:ocp}. For this purpose,
	we require the following additional assumption:
	
	\begin{assumption}\label{assu:stateop}
		We assume that, for every $u \in L^q(0,T;X)$, there exists a unique weak solution $y \in L^q(0,T;\DD)$ 
		of the state equation. The associated solution operator is denoted by $S: L^q(0,T;X) \to L^q(0,T;\DD)$.
	\end{assumption}

    \begin{remark}
        It is to be noted that, without Assumption~\ref{assu:stateop}, the optimal control problem~\eqref{eq:ocp} 
        is actually not really meaningful, since there would be no unique response to a control without this assumption 
        and consequently, it does not make much sense to control the state system. However, in practice, 
        Assumption~\ref{assu:stateop} can be guaranteed by requiring suitable conditions on the nonlinearity $f$. 
        An example for such a setting will be given in Theorem~\ref{thm:gradlip} below, see also Remark~\ref{rem:assus}.	
    \end{remark}

	As an immediate consequence of Corollary~\ref{cor:implfunc} and Theorem~\ref{thm:barBdiff}, we obtain the following 
	result, similarly to the derivation of the optimality system in the previous section:
	
	\begin{corollary}\label{cor:JJdiff}
		The reduced objective defined by 
		\begin{equation*}
			\JJ: \R^n \ni \tau \mapsto J(S(\bar B(\tau)), \tau) \in \R
		\end{equation*}
		is continuously Fr\'echet-differentiable with derivative 
		\begin{equation}\label{eq:gradJ}
			\nabla \JJ(\tau) = 
			\begin{pmatrix}
				- \dual{\psi_1}{P(\tau_1)}_{Z^*, Z} \\
				\big(\dual{\psi_{i-1} - \psi_i}{P(\tau_i)}_{Z^*, Z}\big)_{i=2}^{n-1} \\
				\dual{\psi_{n-1}}{P(\tau_n)}_{Z^*, Z}
			\end{pmatrix} 
			+ \nabla_\tau J(y, \tau),
		\end{equation}
		where $y = S(\bar B(\tau))$ and $p \in \W_T^{q'}(X^*, \DD^*)$ is a strong solution of the adjoint equation 
		\begin{equation*}
			\begin{aligned}
				- p'(t) + A^* p(t) + f'(y(t))^* p(t) &= \partial_y J(y, \tau)(t) 
				& & \text{in } \DD^* \text{ f.a.a.\ } t\in (0,T), \\
				p(T) &= 0 & & \text{in } Y_q^*,    
			\end{aligned} 
		\end{equation*}
		and $P \in C^{0,\alpha}(\R;Z)$ is the extension of $p$ by constant continuation as defined in \eqref{eq:extensiononR}.
	\end{corollary}
	
	Given the reduced objective, we reformulate \eqref{eq:ocp} as the following finite dimensional optimization problem
	\begin{equation*}
		\eqref{eq:ocp} 
		\quad \Longleftrightarrow \quad 
		\left\{\quad 
		\begin{aligned}
			\min_{\tau \in \R^n} \quad & \JJ(\tau) \\
			\text{s.t.}\quad & \tau \in \PP,
		\end{aligned}
		\right.
	\end{equation*}
	which we aim to solve by means of the projected gradient method resp.\ proximal gradient method. Starting from 
	an initial value $\tau^0\in \R^n$, we compute the next iterate by 
	\begin{equation}\label{eq:projgrad}
	    \tau^{k+1} = \Pi_{\PP}\Big( \tau^k - \frac{1}{L_k} \, \nabla \JJ(\tau^k) \Big), \quad k = 0, 1, 2, \ldots,
    \end{equation}		
	where $\Pi_\PP$ denotes the projection on the feasible set $\PP$ and $L_k^{-1} > 0$ is a step size computed 
	with the backtracking procedure from \cite[Section~10.3.3]{Bec17}. This leads to the following algorithm:
	
	\begin{algorithm} \ \\[-3ex]
    
	\begin{algorithmic}[1]
		\STATE Choose $\tau_0\in \PP$ and $\gamma \in (0,1)$. Set $k = 0$.
		\REPEAT
		    \STATE Set $g_k := \nabla \JJ(\tau_k)$.
		    \STATE Set $L_0 := 1$ and $m = 0$.
            \WHILE {$\JJ(\tau_k) - \JJ(\Pi_\PP(\tau_k - L_m^{-1} g_k)) 
            < \gamma\,L_m\, \| \tau_k - \Pi_\PP(\tau_k - L_m^{-1} g_k) \|^2$}
                \STATE Set $L_{m+1} := 2\,L_m$ and update $m \leftarrow m + 1$.
            \ENDWHILE
            \STATE Set $L_k := L_m$, $\tau_{k+1} := \Pi_\PP(\tau_k - L_k^{-1} g_k)$ and update $k \leftarrow k + 1$.
		\UNTIL {$\tau_{k} = \tau_{k-1}$}
	\end{algorithmic}
	\label{alg:projgrad}
    \end{algorithm}
    
The practical realization of this algorithm requires an efficient computation of 
the projection onto the feasible set $\PP$, which is the polyhedron from \eqref{eq:feaspoly}.
This is a classical problem in the literature, also known as isotonic regression, for which 
efficient algorithms are available, see, e.g., \cite{GW84}.

\begin{theorem}\label{thm:proxconv}
    Let Assumption~\ref{assu:stateop} hold true.
    Then, for every initial value $\tau_0 \in \PP$, the proximal gradient method from Algorithm~\ref{alg:projgrad} 
    satisfies the following:
    \begin{enumerate}
        \item Feasibility: The algorithm is feasible, i.e., the backtracking procedure in every outer iteration terminates 
        after finitely many iterations.
        \item Descent property: The sequence $\{\JJ(\tau_k)\}_{k\in\N}$ is non-increasing and
        $\JJ(\tau_{k+1}) < \JJ(\tau_k)$, if and only if $\tau_k$ is not a stationary point of $\eqref{eq:ocp}$.
        
        \item Stationarity: Every accumulation point of $\{\tau_k\}_{k\in\N}$ is a stationary point of \eqref{eq:ocp}
        and thus satisfies the optimality system from Theorem~\ref{thm:KKT}.
    \end{enumerate}     
\end{theorem}    	

\begin{proof}
    We rewrite the optimization problem in the following form:
	\begin{equation*}
        \eqref{eq:ocp}	
        \quad \Longleftrightarrow \quad
        \min_{\tau\in \R^n} \, \JJ(\tau) + I_{\PP}(\tau),
	\end{equation*}
    where $I_\PP$ denotes the indicator functional of $\PP$, i.e., 
    \begin{equation}\label{eq:indicator}
        I_\PP(\tau) := 
        \begin{cases}
            0 , & \tau \in \PP, \\
            \infty, & \tau \notin \PP.
        \end{cases}
    \end{equation}    	
    Applying the proximal gradient method to this problem exactly leads to the iteration in \eqref{eq:projgrad}.
    Now, $\JJ: \R^n \to \R$ is continuously differentiable by Corollary~\ref{cor:JJdiff} and 
    $I_{\PP}: \R^n \to [0,\infty]$ is lower semicontinous and continuous on $\PP = \dom(I_\PP)$.
    Moreover, $\JJ$ is clearly bounded on $\PP$ by boundedness of $\PP$ and continuity of $\JJ$. 
    Therefore, \cite[Theorem~3.1]{KM22} is applicable, which implies the result. 
	Note that the stationarity of a limit point $\bar\tau$, i.e., 
    \begin{equation*}
        \bar\tau \in \PP, \quad 
        \dual{\nabla \JJ(\bar\tau)}{\tau - \bar\tau} \geq 0\quad \forall\, \tau \in \PP,
    \end{equation*}
	indeed implies the KKT-conditions from Theorem~\ref{thm:KKT}, since the starting point 
	for derivation of the latter is exactly the variational inequality in \eqref{eq:VI}.
\end{proof}

If the gradient of the smooth part of the objective is Lipschitz continuous, one can even establish 
an order of convergence for the residuum in the fixed point equation, see Corollary~\ref{cor:proxconv} below.
The Lipschitz continuity is a hard assumption in our context, but it can be verified under additional assumptions on the data, 
as we will prove next:
	
	\begin{theorem}\label{thm:gradlip}
	    Beside our standing assumptions, suppose that the following conditions holds true: 
	    \begin{enumerate}
	        \item The number $q$ from Assumption~\ref{assu:f} satisfies $q\geq 2$.
	        
	        \item The form functions satisfy $\psi_i \in \DD$, $i = 1, \ldots, n-1$, such that 
	        the space $Z$ from \eqref{eq:defZ} is simply $Z= \DD^*$.
	        
	        \item The derivative of the objective is Lipschitz-continuous on bounded sets
	        in the following sense: For every $M > 0$, 
	        there exists $L^J_M > 0$ such that  
            \begin{equation}\label{eq:Jprimelip}
            \begin{aligned}
                & \| J'(y_1, \tau_1) - J'(y_2, \tau_2)\|_{L(L^q(0,T;\DD)\times \R^n), \R)} \\
                & \qquad\qquad 
                \leq L^J_M\, \big(\| y_1 - y_2\|_{L^q(0,T;\DD)} + \|\tau_1 - \tau_2\|_{\R^n}\big)\\
                & \qquad \forall\, y_i \in W^q_0(\DD,X),\,\tau \in \R^n \; : \; 
                \|y_i\|_{\W^q(\DD,X)} + \|\tau_i\|_{\R^n} \leq M, \; i =1,2 .
            \end{aligned}
            \end{equation}
            
            \item Furthermore, $\partial_y J$ maps bounded sets of $\W^q(\DD,X) \times \R^n$ 
            to bounded sets of $C([0,T];Y_q^*)$. \label{it2} 
            
    	        \item\label{it5} The nonlinearity $f$ maps $Y_q$ to $X$ and        
    	        is globally Lipschitz continuous in these spaces, i.e., 
    	        there is a constant $L_f > 0$ such that 
	        \begin{equation*}
	        \begin{aligned}
                \| f(y_1) - f(y_2)\|_{X}  \leq L_f \, \| y_1 - y_2\|_{Y_q} 
                \quad \forall\,y_1, y_2\in Y_q.
	        \end{aligned}
	        \end{equation*}
	        
	        \item\label{it6} The derivative of $f$ is Lipschitz-continuous in the following sense: 
	        There exist a constant $K_f > 0$ and an exponent $s \in [0, q-2]$ such that 
	        \begin{equation}\label{eq:fprimelip}
	            \| f'(y_1) - f'(y_2) \|_{L(\DD;Y_q)} 
	            \leq K_f \big( \|y_1\|_\DD^s + \|y_2\|_\DD^s\big) \, \|y_1 - y_2 \|_\DD
	            \quad \forall\, y_1, y_2 \in \DD.
	        \end{equation}
	        
	        \item\label{it7} Moreover, $f'$ maps $Y_q$ to $L(Y_q, Y_q)$ and is continuous in these spaces.
	    \end{enumerate}
	    Then, for every $\tau\in \R^n$, 
	    the state equation \eqref{eq:stateeq} admits a unique solution in $\W^q_0(\DD,X)$ and 
	    the mapping $\R^n \ni \tau \mapsto \nabla \JJ(\tau) \in \R^n$ is Lipschitz continuous on bounded sets, i.e., 
	    for every $M > 0$, there exists $L_M > 0$ such that 
	    \begin{equation*}
	        \|\nabla\JJ(\tau_1) - \nabla\JJ(\tau_2)\|_{\R^n} \leq L_M \,\|\tau_1 - \tau_2 \|_{\R^n} \quad 
	        \forall\, \tau_1, \tau_2 \in \overline{B_{\R^n}(0,M)}.
	    \end{equation*}
	\end{theorem}
	
\begin{proof}
    We split the rather lengthy proof into several steps beginning with the existence of the state equation, 
    which is a straight forward consequence of Banach's contraction principle, similarly to the proof of 
    Lemma~\ref{lem:adjointeq}. Nonetheless, let us present the proof in detail for convenience of the reader 
    and for later reference. 

    \vspace*{1ex}
    \emph{Step 1: Existence of solutions of the state equation}\\
    Let $u \in L^q(0,T;X)$ be fixed but arbitrary. We follow the lines of the proof of Lemma~\ref{lem:adjointeq}.
    For this purpose, let $t \in (0,T]$ be arbitrary and consider the auxiliary equation 
    \begin{equation}\label{eq:stateeqfoo}
        y' + A y + f(y) = u \quad \text{in } (0,t), \quad y(0) = y_0.
    \end{equation}
    with $y_0 \in Y_q$ given. Let us denote the solution operator associated  with the linear part by
    \begin{equation}\label{eq:defG2}
	    G_{t} := (\partial_t + A, \gamma_0)^{-1} :  L^{q}(0,t;X) \times Y_q \to \W^{q}(0,t; \DD, X)
	\end{equation}
	with $\W^{q}(0, t; \DD, X) := W^{1,q}(0,t;X) \cap L^q(0,b;\DD)$.
    Note that $G_t$ is well defined due to the maximal parabolic regularity of $A$. 
    Due to $\W^{q}(0,t; \DD, X) \embed C([0,t];Y_q)$, we can consider $G_{t}$ as an  operator with range in 
    $C([0,t];Y_q)$ and we abbreviate the associated operator norm by $\|G_{t}\|$. 
    Similar to the proof of Lemma~\ref{lem:adjointeq}, we introduce the fixed point mapping
	\begin{equation*}
	    \Psi_{t; u, y_0}: C([0,t]; Y_q) \ni y \mapsto G_{t}(u - f(y), y_0) \in C([0,t]; Y_q)
	\end{equation*}
	and prove its contractivity on small time intervals. Indeed, for arbitrary $y_1, y_2 \in C([0,t]; Y_q)$, 
	by using the assumed Lipschitz continuity of $f$, an estimate similar to \eqref{eq:Gest} is obtained:
	\begin{equation}
	\begin{aligned}
        \|\Psi_{t;u, y_0}(y_1) - \Psi_{t; u, y_0}(y_2)\|_{C([0,t]; Y_q)} 
        &  \leq \|G_{t}\| \, \| f(y_1) - f(y_2)\|_{L^{q}(0,t;X)}\\
        & \leq \sup_{t\in [0,T]} \|G_t\| \,L_f \, t^{1/q} \, \|y_1 - y_2\|_{C([0,t]; Y_q)} \\
        & \leq C_T\, L_f \, t^{1/q} \, \|y_1 - y_2\|_{C([0,t]; Y_q)}
    	\end{aligned}
	\end{equation}
    where we applied Lemma~\ref{lem:normGt} from the appendix for the last inequality.
    This shows that $\Psi_{t;u, y_0}$ is indeed contractive for $t > 0$ sufficiently small. A standard concatenation argument 
    implies the existence of a solution on the whole time interval $[0,T]$. Moreover, due to
    \begin{equation*}
    \begin{aligned}
        & \|y\|_{\W^q(\DD, X)} \\
        & \quad \leq \|G_T\|_{L(L^{q}(0,t;X) \times Y_q, \W^{q}(\DD, X))}\, 
        \big( \|u\|_{L^q(0,T;X)} + \|f(y)\|_{L^q(0,T;X)} + \|y_0\|_{Y_q}\big) \\
        & \quad \leq C ( 1 + L_f \,\|y\|_{L^q(0,T;Y_q)} + \|f(0)\|_{L^q(0,T;X)}\big),
    \end{aligned}
    \end{equation*}     
    we find that the solution is an element in $\W^{q}(\DD, X)$. 
    
    Next, we prove that the solution operator associated with the state equation is 
    globally Lipschitz continuous when considered as a mapping from $L^q(0,T;X)$ to $\W^q(\DD, X)$.    
    For this purpose, let $u_1, u_1 \in L^q(0,T;X)$ 
    and $y_0^1, y_0^2 \in Y_q$ be given and define $\bar t := (2 C_T L_f)^{-q}$.
    Then the solutions $y_1, y_2 \in \W^q(0,\bar t;\DD, X)$ to \eqref{eq:stateeqfoo} 
    with $t = \bar t$ and $u = u_i$, $i=1,2$, satisfy
    \begin{equation*}
    \begin{aligned}
        \| y_1 - y_2 \|_{C([0,\bar t]; Y_q)}
        & \leq \|\Psi_{\bar t; u_1, y_0^1}(y_1) - \Psi_{\bar t; u_2, y_0^2}(y_1) \|_{C([0,\bar t]; Y_q)} \\
        & \quad + \|\Psi_{\bar t; u_2, y_0^2}(y_1) - \Psi_{\bar t; u_2, y_0^2}(y_2) \|_{C([0,\bar t]; Y_q)} \\
        & \leq \|G_{\bar t}\| \big( \|u_1 - u_2\|_{L^q(0,\bar t; X)} + \|y_0^1 - y_0^2\|_{Y_q}\big) \\
        & \quad + \|G_{\bar t}\| \,L_f \, \bar t^{1/q} \, \|y_1 - y_2 \|_{C([0,\bar t]; Y_q)}.
    \end{aligned}
    \end{equation*}        
    and Lemma~\ref{lem:normGt} and the definition of $\bar t$ imply
    \begin{equation}\label{eq:Slipbart}
        \| y_1 - y_2 \|_{C([0,\bar t]; Y_q)} \leq 2\,C_T \big( \|u_1 - u_2\|_{L^q(0,\bar t; X)} + \|y_0^1 - y_0^2\|_{Y_q}\big) .
    \end{equation}        
    Consider now \eqref{eq:stateeqfoo} with $t = T$ and $u = u_i$, $i=1,2$, (but the same initial value $y_0$)
    and denote the solutions 
    again by $y_i$, $i=1,2$. Let moreover $t\in [0,T]$ be arbitrary and assume w.l.o.g., that $T = N \bar t$ 
    with some $N \in \N$. Then there is an $n\in \N$ such that $t\in [n\bar t, (n+1)\bar t]$ and 
    a recursive application of \eqref{eq:Slipbart} gives
    \begin{equation*}
    \begin{aligned}
        \|y_1(t) - y_2(t)\|_{Y_q} 
        & \leq \| y_1 - y_2 \|_{C([n\bar t, (n+1)\bar t]; Y_q)} \\
        & \leq 2 \, C_T \big( \|u_1 - u_2\|_{L^q(n\bar t, (n+1)\bar t; X)} + \| y_1(n\bar t) - y_2(n\bar t)\|_{Y_q}\big) \\
        & \leq 2 \, C_T \big( \|u_1 - u_2\|_{L^q(n\bar t, (n+1)\bar t; X)} + \| y_1 - y_2\|_{{C([(n-1)\bar t, n\bar t]; Y_q)}}\big) \\
        & \leq \ldots 
        \begin{aligned}[t]
            & \leq \sum_{k=0}^{n} (2\,C_T)^{n-k+1} \|u_1 - u_2\|_{L^q(k\bar t, (k+1)\bar t; X)} \\[-1.5ex]
            & \leq (2\,C_T)^{N+1}  \sum_{k=0}^{N} \|u_1 - u_2\|_{L^q(k\bar t, (k+1)\bar t; X)} \\
            & \leq \tilde L_S\, \|u_1 - u_2\|_{L^q(0, T; X)} 
        \end{aligned}                
    \end{aligned}     
    \end{equation*}
    with a constant $\tilde L_S > 0$ independent of $u_1, u_2$.
    This implies the global Lipschitz continuity of the solution operator of \eqref{eq:stateeqfoo} on $(0,T)$ as a
    mapping from $L^q(0,T;X)$ to $C([0,T];Y_q)$. With a slight abuse of notation, we denote this mapping 
    by $S : u \mapsto y$, too. Using a simple boot strapping argument, one shows that $S$ is even globally Lipschitz 
    continuous, when considered with range in the maximal parabolic regularity space, since
    \begin{equation}\label{eq:Slipmaxparreg}
    \begin{aligned}
        \|y_1 - y_2\|_{\W^q(\DD,X)} 
        & \leq \|G_T\| \big( \|f(y_1) - f(y_2)\|_{L^q(0,T;X)} + \|u_1 - u_2\|_{L^q(0,T;X)}\big) \\
        & \leq C_T \big( L\,\|y_1 - y_2\|_{L^q(0,T;Y_q)}  + \|u_1 - u_2\|_{L^q(0,T;X)} \big) \\
        & \leq C_T \big( L\, \tilde L_S + 1\big) \|u_1 - u_2\|_{L^q(0,T;X)} \\
        & =: L_S\,  \|u_1 - u_2\|_{L^q(0,T;X)} .
    \end{aligned}
    \end{equation}

    \vspace*{1ex}
    \emph{Step 2: Boundedness of the adjoint solution operator}\\
    Let $t\in (0,T)$, $g\in L^{q'}(0,T;\DD^*)$, $p_t \in X^*$, and 
    \begin{equation}\label{eq:boundy}
        y \in \overline{B_{L^q(0,T;\DD)}(0, \varrho)}
    \end{equation}        
    be fixed but arbitrary. We first consider an auxiliary adjoint equation given by
    \begin{equation}\label{eq:adjeq2}
	    -p' + A^* p + f'(y)^* p = g, \quad p(t) = p_t.
    \end{equation}
    By Lemma~\ref{lem:adjointeq}, this equation
    admits a unique solution $p \in \W^{q'}(0,t;X^*, \DD^*)$.
    To show the boundedness of the associated solution operator, 
    let us return to the notion of the proof of Lemma~\ref{lem:adjointeq} and introduce the fixed point mapping 
    \begin{equation*}
	    \Phi_{t;g, p_t}: C([0,t]; Y_q^*) \ni p \mapsto K_t(g - f'(y)^*p, p_t) \in C([0,T]; Y_q^*),
    \end{equation*}
    where $K_t$ is as defined in \eqref{eq:defG}. From \eqref{eq:fprimestargrowth} and Lemma~\ref{lem:normGt}, 
    it follows that
    \begin{equation*}
    \begin{aligned}
         \|p_1 - p_2\|_{C([0,t];Y_q^*)}
        & = \|\Phi_{t; g_1, p_t^1}(p_1) - \Phi_{t; g_1, p_t^1}(p_1)\|_{C([0,t];Y_q^*)} \\
        &  \leq \|K_t\| \|f'(y)^*\|_{L^r(0,T; L(Y_q^*, \DD^*))} \, t^{\frac{r - q'}{r q'}} \|p_1 - p_2\|_{C([0,t];Y_q^*)} \\
        & \quad + \|K_t\| \|g_1 - g_2\|_{L^{q'}(0,t;\DD^*) + \|p_t^1 -p_t^2\|_{Y_q^*}} \\
        & \leq C_T 
        \begin{aligned}[t]
            \Big[ & \big( d_1 \, T^{1/r} + d_2 \, \|y\|_{L^q(0,T;\DD)}^{q/r}\big)\, t^{\frac{r - q'}{r q'}} \|p_1 - p_2\|_{C([0,t];Y_q^*)}\\
            & +  \|g_1 - g_2\|_{L^{q'}(0,t;\DD^*)} + \|p_t^1 -p_t^2\|_{Y_q^*}\Big] .
        \end{aligned}
    \end{aligned}
    \end{equation*}
    Therefore, if we choose $\bar t := (2C^T(d_1 \, T^{1/r} + d_2 \, \varrho^{q/r}))^{-rq'/(r-q')}$, then, 
    analogously to \eqref{eq:Slipbart}, 
    \begin{equation*}
        \|p_1 - p_2\|_{C([0,\bar t];Y_q^*)}
        \leq 2\,C_T \big( \|g_1 - g_2\|_{L^{q'}(0,t;\DD^*)} + \|p_t^1 -p_t^2\|_{Y_q^*}\big)
    \end{equation*}
    is obtained. Thus we can argue as in case of the state equation to obtain the existence of a constant $L_{P} > 0$ 
    such that
    \begin{equation*}
        \| p_1 - p_2 \|_{C([0,T];Y_q^*)} \leq L_P\, \|g_1 - g_2\|_{L^{q'}(0,T;\DD^*)}.
    \end{equation*}
    Since $r > q'$ by Assumption~\ref{assu:f}, we can apply a boot strapping argument to obtain the same result in 
    the maximal parabolic regularity space:
    \begin{equation*}
    \begin{aligned}
        & \|p_1 - p_2\|_{\W^{q'}(X^*;\DD^*)} \\
        & \quad \leq \|K_T\|_{L(L^{q'}(0,T;\DD^*),\W^{q'}(X^*;\DD^*))} 
        \| g_1 - g_2 - f'(y)^* (p_1 - p_2)\|_{L^{q'}(0,T;\DD^*)} \\
        & \quad \leq C_T \Big( \|g_1 - g_2\|_{L^{q'}(0,T;\DD^*)}  
        + C\,\| f'(y)^*\|_{L^{r}(0,T;L(Y_q^*, \DD^*))} \|p_1 - p_2\|_{C([0,T];Y_q^*)}  \Big) \\
        & \quad \leq C \big(1 + C(d_1 \,T^{1/q'} + d_2\,\varrho^{q/r}) L_P \big) \|g_1 - g_2\|_{L^{q'}(0,T;\DD^*)} \\
        & \quad =: C^*_\varrho\, \|g_1 - g_2\|_{L^{q'}(0,T;\DD^*)} .
    \end{aligned}
    \end{equation*}
    Since the adjoint equation is linear, this implies the boundedness of its solution operator 
    with the same constant, i.e.,
    \begin{equation}\label{eq:boundadj}
        \| (- \partial_t + A^* + f'(y)^*)^{-1}\|_{L(L^{q'}(0,T;\DD^*),\W^{q'}(X^*;\DD^*))} \leq C^*_\varrho.
    \end{equation}
  
      \vspace*{1ex}      
    \emph{Step 3: Lipschitz continuity of the state solution operator in weaker spaces}\\
    In the first step of the proof, we have already seen that the solution operator $S$ of the state equation is 
    globally Lipschitz continuous as mapping from $L^q(0,T;X)$ to $\W^q(\DD,X)$. However, since the 
    switching-point-to-control operator fails to be Lipschitz continuous as a mapping with range in $L^q(0,T;X)$, 
    as we will see below, we need the Lipschitz continuity of $S$ in weaker spaces. For this purpose, 
    let $u_1, u_2 \in L^q(0,T;X)$ be arbitrary with 
    \begin{equation*}
        \| u_i \|_{L^q(0,T;X)} \leq R, \quad i = 1,2.
    \end{equation*}
    By \eqref{eq:Slipmaxparreg}, we know that there is a constant $\varrho_R > 0$ such that 
    \begin{equation}\label{eq:boundy2}
        \|S(u)\|_{\W^q(\DD, X)} \leq \varrho_R \quad \forall \, u \in \overline{B_{L^q(0,T;X)}(0,R)}.
    \end{equation}
    From Corollary~\ref{cor:implfunc} it follows that    
    \begin{equation*}
    \begin{aligned}
        \|S(u_1) - S(u_2) \|_{L^q(0,T;\DD)}
        \leq \sup_{t\in (0,1)} \|S'(u_1+t\,(u_2 - u_1))(u_2 - u_2)\|_{L^q(0,T;\DD)}.
    \end{aligned}
    \end{equation*}
    Moreover, Corollary~\ref{cor:implfunc} tells us that, for every $t\in (0,1)$, the derivative of $S$ 
    at $u_1+t\,(u_2 - u_1)$ can be extended to an operator with domain in $[\W_T^{q'}(X^*, \DD^*)]^*$
    and thus, with the notation of that corollary, we obtain
    \begin{equation*}
    \begin{aligned}
        & \|S'(u_1+t\,(u_2 - u_1))(u_2 - u_2)\|_{L^q(0,T;\DD)}\\
        & \quad = \|\bar S'(u_1+t\,(u_2 - u_1))(u_2 - u_2)\|_{L^q(0,T;\DD)} \\
        & \quad \leq \|\bar S'(u_1+t\,(u_2 - u_1))\|_{L([\W_T^{q'}(X^*, \DD^*)]^*, L^{q}(0,T;\DD))} 
         \| u_2 - u_1 \|_{[\W_T^{q'}(X^*, \DD^*)]^*} .
    \end{aligned}
    \end{equation*}  
    For the operator norm of the derivative, \eqref{eq:boundadj} along with \eqref{eq:boundy2} implies
    \begin{equation*}
    \begin{aligned}
        & \|\bar S'(u_1+t\,(u_2 - u_1))\|_{L([\W_T^{q'}(X^*, \DD^*)]^*, L^{q}(0,T;\DD))} \\
        & \quad = \|\bar S'(u_1+t\,(u_2 - u_1))^*\|_{L(L^{q'}(0,T;\DD^*),\W^{q'}_T(X^*;\DD^*))} \\
        & \quad = \| (- \partial_t + A^* + f'(S(u+t\,h))^*)^{-1}\|_{L(L^{q'}(0,T;\DD^*),\W^{q'}_T(X^*;\DD^*))} 
        \leq C^*_{\varrho_R}
    \end{aligned}
    \end{equation*}
    and thus 
    \begin{equation}\label{eq:Sweaklip}
        \|S(u_1) - S(u_2) \|_{L^q(0,T;\DD)}
        \leq C^*_{\varrho_R} \,\|u_2 - u_1 \|_{[\W_T^{q'}(X^*, \DD^*)]^*}  
    \end{equation}
    for all $u_1, u_2 \in \overline{B_{L^q(0,T;X)}(0,R)}$.
    
     \vspace*{1ex}
\emph{Step 4: H\"older and Lipschitz continuity of the switching-point-to-control operator}\\
    Consider the switching-point-to-control operator as a mapping with 
    range in $L^q(0,T;X)$, for simplicity also denoted by $\bar B$, i.e.,
    \begin{equation*}
        \bar B: \R^n \ni \tau \mapsto u := \sum_{i=1}^{n-1}\bar\chi_{[\tau_i, \tau_{i+1})}(t) \, \psi_i \in L^q(0,T;X).
    \end{equation*}
    Thanks to $Z^* = (\DD, X)_{\vartheta, \infty}\embed X$, we deduce from \eqref{eq:barchieq} along with 
    H\"older's inequality and the triangle inequality that
    \begin{equation*}
    \begin{aligned}
        & \|\bar B(\tau_1) - \bar B(\tau_2)\|_{L^q(0,T;X)}\\
        & \quad \leq \Big(\int_0^T \Big(\sum_{i=1}^{n-1} 
        \big| \bar\chi_{[\tau_i^1, \tau_{i+1}^1)}(t) - \bar\chi_{[\tau_i^2, \tau_{i+1}^2)}(t)\big| \|\psi_i\|_{X}\Big)^q\,\d t\Big)^{\frac{1}{q}} \\
        & \quad \leq \Big(\sum_{i=1}^{n-1}\|\psi_i\|_{X}^{q'}\Big)^{\frac{1}{q'}}
        \Big(\int_0^T 
        \sum_{i=1}^{n-1}\big| \bar\chi_{[\tau_{i+1}^2, \tau_{i+1}^1)}(t) - \bar\chi_{[\tau_i^2, \tau_i^1)}(t)\big|^q\,\d t\Big)^{\frac{1}{q}} \\
        & \quad \leq \Big(\sum_{i=1}^{n-1}\|\psi_i\|_{X}^{q'}\Big)^{\frac{1}{q'}}
        \sum_{i=1}^{n-1}  \big( \| \bar\chi_{[\tau_{i+1}^2, \tau_{i+1}^1)} \|_{L^q(0,T)}
        + \|  \bar\chi_{[\tau_i^2, \tau_i^1)}\|_{L^q(0,T)} \big) \\
        & \quad \leq C \Big(\sum_{i=1}^{n-1}\|\psi_i\|_{Z^*}^{q'}\Big)^{\frac{1}{q'}}
        \sum_{i=1}^{n-1}  \Big(|\tau^1_i - \tau^2_i|^{\frac{1}{q}} + |\tau^1_{i+1} - \tau^2_{i+1}|^{\frac{1}{q}}\Big)
    \end{aligned}
    \end{equation*}
    and consequently, $\bar B$ is globally H\"older continuous with H\"older exponent $1/q$.
    Note that this also illustrates the lack of differentiability of $\bar B$ with values in $L^q(0,T;X)$. 
    In its original spaces, i.e., as the operator defined in \eqref{eq:Bweak2}, $\bar B$ is 
    continuously differentiable as seen in Theorem~\ref{thm:barBdiff} and it is globally Lipschitz continuous 
    in these spaces, too, since the form of the derivative in \eqref{eq:barBabl} implies
    \begin{equation*}
    \begin{aligned}
        & \|\bar B'(\tau)\|_{L(\R^n;[C^{0,\alpha}(0,T;Z))]^*)} \\
        &\quad \leq \sup_{\stackrel{\|h\|_{\R^n} \leq 1}{\|v\|_{C^{0,\alpha}(0,T;Z))} \leq 1}} 
        \sum_{i=1}^{n-1}  \|\psi_i\|_{Z^*} \Big( |h_{i+1}| \|v(\Pi_{[0,T]}(\tau_{i+1}))\|_Z + |h_{i}| \|v(\Pi_{[0,T]}(\tau_{i}))\|_Z \Big) \\
        &\quad \leq 2 \Big( \sum_{i=1}^{n-1} \|\psi_i\|_{Z^*}^2 \Big)^{1/2},
    \end{aligned}
    \end{equation*}
    where $\Pi_{[0,T]}(t) := \max\{ 0; \min\{T; t\}\}$.

    \vspace*{1ex}
    \emph{Step 5: Lipschitz continuity of the adjoint solution operator}\\
    Let $y_1, y_2 \in \overline{B_{L^q(0,T;\DD)}(0, \varrho)}$ and $g_1, g_2 \in L^{q'}(0,T;\DD^*)$ be given and 
    consider the equations
    \begin{equation*}
	    -p_i' + A^* p_i + f'(y_i)^* p_i = g_i, \quad p_i(T) = 0, \quad i=1,2
    \end{equation*}
    so that the difference solves
    \begin{multline*}
        - (p_1 - p_1)' + A^*(p_1 - p_2) + f'(y_1)^*(p_1 - p_2)  \\[-1ex]
        = g_1 - g_2 + \big( f'(y_2)^* - f'(y_1)^* \big) p_2.
    \end{multline*}
    Hence, \eqref{eq:boundadj} implies
    \begin{equation*}
    \begin{aligned}
        \|p_1 - p_2\|_{\W^{q'}(X^*;\DD^*)}  
        & \leq C^*_{\varrho} \,
        \begin{aligned}[t]
            \big( & \| g_1 - g_2 \|_{L^{q'}(0,T;\DD^*)} \\
            &+ \|f'(y_2)^* - f'(y_1)^*\|_{L^{q'}(0,T; L(Y_q^*,\DD^*))} \|p_2\|_{C([0,T];Y_q^*)} \big) 
        \end{aligned}\\
        & \leq C^*_\varrho
        \begin{aligned}[t]
            \big( & \| g_1 - g_2 \|_{L^{q'}(0,T;\DD^*)} \\
            &+ C^*_\varrho \, \|f'(y_2) - f'(y_1)\|_{L^{q'}(0,T; L(\DD, Y_q))} \|g_2\|_{L^{q'}(0,T;\DD^*)} \big) .         
        \end{aligned} 
    \end{aligned}
    \end{equation*}
    Together with the assumption in \eqref{eq:fprimelip}, 
    H\"older's inequality with $\frac{q'}{q} + \frac{q - q'}{q} = 1$, keeping in mind that $s\,q' \,\frac{q}{q - q'} = s\,\frac{q}{q-2} \leq q$, gives
    for the right hand side
    \begin{equation*}
    \begin{aligned}
        & \|f'(y_2) - f'(y_1)\|_{L^{q'}(0,T; L(\DD, Y_q))} \\
        & \qquad \leq K_f \Big( \int_0^T (\|y_1\|_\DD^s +  \|y_2\|_\DD^s)^{q'} \|y_1 - y_2\|_\DD^{q'}\,\d t \Big)^{1/q'} \\        
        & \qquad \leq C\big( \|y_1\|_{L^q(0,T;\DD)}^s + \|y_1\|_{L^q(0,T;\DD)}^s \big) \|y_1 - y_2\|_{L^q(0,T;\DD)} \\
        & \qquad \leq 2 \, C\, \varrho^s\,\|y_1 - y_2\|_{L^q(0,T;\DD)}
    \end{aligned}        
    \end{equation*}
    and consequently, 
    \begin{equation}\label{eq:plip}
    \begin{aligned}
        & \|p_1 - p_2\|_{\W^{q'}(X^*;\DD^*)}  \\
        & \quad \leq L_{P, \varrho} \big( 
         \| g_1 - g_2 \|_{L^{q'}(0,T;\DD^*)} + \|y_1 - y_2\|_{L^q(0,T;\DD)}\, \|g_2\|_{L^{q'}(0,T;\DD^*)} \big)
    \end{aligned}
    \end{equation}
    with a Lipschitz constant $L_{P,\varrho} > 0$ depending on $\varrho$. 

    \vspace*{1ex}
    \emph{Step 6: Lipschitz continuity of the gradient of the objective}\\
    Let $\tau_1, \tau_2 \in \overline{B_{\R^n}(0,M)}$ and $j \in \{1,\ldots, n\}$ be arbitrary. 
    Recall the form of the gradient of the objective in \eqref{eq:gradJ}. Concerning the Lipschitz continuity of the 
    first addend, thanks to $Z = \DD^*$ by assumption, it suffices to establish
    the existence of a constant $L_M^{(j)}>0$ such that 
    \begin{equation}\label{eq:Pdist0}
        \|P_1(\tau^1_j) - P_2(\tau^2_j)\|_{\DD^*} \leq L_M^{(j)} \,\|\tau_1 - \tau_2\|_{\R^n},
    \end{equation}
    where $P_1$ and $P_2$ denote the extensions by constant continuation according to \eqref{eq:extensiononR} 
    of $p_1, p_2 \in \W_T^{q'}(X^*, \DD^*)$, which are the solutions of 
    \begin{equation}
	    -p_i' + A^* p_i + f'(y_i)^* p_i = \partial_y J(y_i, \tau_i), \quad p_i(T) = 0, \quad i=1,2
    \end{equation}
    and $y_i = S(\bar B(\tau_i)) \in \W^q_0(\DD, X)$, $i=1,2$, solves
    \begin{equation}
        y_i' + A y_i + f(y_i) = \bar B(\tau_i), \quad y(0) = y_0,\quad i = 1,2.
    \end{equation}
    Let us define 
    \begin{equation*}
        \bar \tau_j^i := \max\{ 0; \min\{T; \tau^i_j\}\}, \quad i = 1, 2 
    \end{equation*}
    \begin{equation}\label{eq:Pdist}
        \|P_1(\tau^1_j) - P_2(\tau^2_j)\|_{\DD^*} 
        \leq \|p_1(\bar\tau^1_j) - p_2(\bar\tau^1_j)\|_{\DD^*}
        + \|p_2(\bar\tau^1_j) - p_2(\bar\tau^2_j)\|_{\DD^*}
    \end{equation}
    Since $\|\tau_i\|_{\R^n} \leq M$, the global H\"older continuity of $\bar B$ by step~4 and $\bar B(0) = 0$ imply
    \begin{equation}
        \|\bar B(\tau_i)\|_{L^q(0,T;X)} \leq \bar K \, M^{1/q} =: R,
    \end{equation}
    where $\bar K > 0$ is the H\"older constant of $\bar B: \R^n \to L^q(0,T;X)$. 
    Moreover, from \eqref{eq:Slipmaxparreg}, we deduce 
    \begin{equation}\label{eq:ybound}
    \begin{aligned}
         \|y_i\|_{\W^q(\DD, X)} 
        & \leq \big( L_S \, \|\bar B(\tau_i)\|_{L^q(0,T;X)} + \|S(0)\|_{\W^q(\DD, X)} \big)\\
        & \leq \big( L_S\, R + \|S(0)\|_{\W^q(\DD, X)} \big) =: \varrho
    \end{aligned}
    \end{equation}        
    and so, \eqref{eq:boundadj} and \eqref{eq:Jprimelip} yield
    \begin{equation}\label{eq:pbound}
        \|p_i\|_{\W^{q'}(X^*, \DD^*)} 
        \leq C^*_\varrho \, \| \partial_y J(y_i, \tau_i) \|_{L^{q'}(0,T;\DD^*)}
        \leq C^*_\varrho \, L^J_{M + \varrho} (\varrho + M) .
    \end{equation}
    Owing to \eqref{eq:ybound} and \eqref{eq:pbound} and the continuous embeddings 
    $\W^{q}(\DD, X) \embed C([0,T];Y_q)$ and $\W^{q'}(X^*, \DD^*) \embed C([0,T];Y_q^*)$,  
    Lemma~\ref{lem:adjlip} in the appendix together with 
    the mapping properties of $\partial_y J$ from the assumption in \eqref{it2}
    and the assumed continuity of $f'$ from $Y_q$ to $L(Y_q, Y_q)$ implies
    \begin{equation*}
    \begin{aligned}
        \|p_2\|_{C^{0,1}(0,T;\DD^*)} 
        & \leq C' \| \partial_y J(y_2, \tau_2) -  f'(y_2)^*p_2\|_{C([0,T];Y_q^*)}\\
        & \leq C' \| \partial_y J(y_2, \tau_2) \|_{C([0,T];Y_q^*)} \\
        & \qquad+  \sup_{t\in [0,T]} \|f'(y_2(t))\|_{L(Y_q, Y_q)}
        \|p_2\|_{C([0,T];Y_q^*)}
        \leq C_p
    \end{aligned}
    \end{equation*}
    with a constant $C_p > 0$, independent of $\tau_1$ and $\tau_2$.
    Therefore, we obtain for the second term in \eqref{eq:Pdist} that
    \begin{equation}\label{eq:pdistlip}
        \|p_2(\bar\tau^1_j) - p_2(\bar\tau^2_j)\|_{\DD^*} 
        \leq C_p \, \|\tau_1 - \tau_2\|_{\R^n}.
    \end{equation}
    For the first term on the right hand side of \eqref{eq:Pdist}, we can apply \eqref{eq:plip}, 
    which results in 
    \begin{equation}\label{eq:pdist1}
    \begin{aligned}
         & \|p_1 - p_2\|_{C([0,T];\DD^*)}  \\
        & \qquad \leq C\,\|p_1 - p_2\|_{\W^{q'}(X^*;\DD^*)}  \\
        & \qquad \leq C\,L_{P, \varrho} 
        \begin{aligned}[t]
             \big( & \| \partial_y J(y_1, \tau_1) - \partial_y J(y_2, \tau_2) \|_{L^{q'}(0,T;\DD^*)} \\
             & + \|y_1 - y_2\|_{L^q(0,T;\DD)}\, \|\partial_y J(y_2, \tau_2)\|_{L^{q'}(0,T;\DD^*)} \big) .
        \end{aligned}         
    \end{aligned}
    \end{equation}
    To estimate the right hand side, note first that the Lipschitz continuity of $S$ according to \eqref{eq:Sweaklip} 
    along with the Lipschitz continuity of $\bar B$ from  $\R^n$ to 
    $[C^{0,\alpha}(0,T;\DD^*)]^* \embed [\W^{q'}_T(X^*, \DD^*)]^*$ yields 
    \begin{equation}\label{eq:ydist1}
    \begin{aligned}
        \| y_1 - y_2 \|_{L^q(0,T;\DD)} 
        &= \| S(\bar B(\tau_1)) - S(\bar B(\tau_2)) \|_{L^q(0,T;\DD)} \\
        & \leq C^*_\varrho\, \|\bar B(\tau_1) - \bar B(\tau_2)\|_{[\W^{q'}_T(X^*, \DD^*)]^*}   
        \leq C^*_\varrho L_B \| \tau_1 - \tau_2\|_{\R^n},
    \end{aligned}
    \end{equation}
    where $L_B > 0$ denotes the Lipschitz constant of $\bar B$ as a mapping with range in $[\W^{q'}_T(X^*, \DD^*)]^*$. 
    Using this estimate together, the first term on the right hand side of \eqref{eq:pdist1} 
    is estimated by means of \eqref{eq:Jprimelip} as follows
    \begin{equation*}
    \begin{aligned}
        \| \partial_y J(y_1, \tau_1) - \partial_y J(y_2, \tau_2) \|_{L^{q'}(0,T;\DD^*)}
        \leq L^J_{\varrho + M} ( C^*_\varrho L_B + 1) \|\tau_1 - \tau_2\|_{\R^n} .
    \end{aligned}
    \end{equation*}
    Inserting this together with \eqref{eq:ydist1} and an estimate of $\partial_y J(y_2, \tau_2)$ analogous to 
    \eqref{eq:pbound} in \eqref{eq:pdist1} leads to 
    \begin{equation*}
        \|p_1 - p_2\|_{C([0,T];\DD^*)}
        \leq  C \,L_{P, \varrho} \,L^J_{\varrho + M} \,( C^*_\varrho L_B + 1 + C^*_\varrho L_B(\varrho + M)) 
        \|\tau_1 - \tau_2\|_{\R^n} .
    \end{equation*}
    Along with \eqref{eq:pdistlip} and \eqref{eq:Pdist}, this finally gives
    \begin{equation*}
        \|P_1(\tau^1_j) - P_2(\tau^2_j)\|_{\DD^*} \leq L_M^{(j)} \,\|\tau_1 - \tau_2\|_{\R^n} ,
    \end{equation*}
    with a constant $L_M^{(j)} > 0$ that depends on $M$, but not on $\tau_1$ and $\tau_2$.
    This establishes \eqref{eq:Pdist0} as desired.
    
    For the Lipschitz continuity of the second part of $\nabla \JJ$, i.e., $\partial_\tau J(y, \tau)$, 
    we employ \eqref{eq:Jprimelip} and again \eqref{eq:ydist1} to deduce
    \begin{equation*}
    \begin{aligned}
        \|\partial_\tau J(y_1, \tau_1) - \partial_\tau J(y_2, \tau_2)\|_{\R^n}  
        & \leq L_{\varrho + M}^J \big( \|y_1 - y_2\|_{L^q(0,T;D)} + \|\tau_1 - \tau_2\|_{\R^n}\big) \\
        & \leq L_{\varrho + M}^J ( C^*_\varrho\, L_B + 1)\|\tau_1 - \tau_2\|_{\R^n} ,
    \end{aligned}
    \end{equation*}
    which finally completes the proof.
\end{proof}		
	
\begin{corollary}\label{cor:proxconv}
    Assume that the additional assumptions from Theorem~\ref{thm:gradlip} hold true.
    Denote the residuum in the fixed point equation by
    \begin{equation*}
        r_k := \| \Pi_{\PP} (\tau_k - L_k^{-1}\nabla \JJ(\tau_k)) - \tau_k \|_{\R^n},
    \end{equation*}
    where $\{\tau^k\}_{k\in \N}$ is the sequence generated by Algorithm~\ref{alg:projgrad}.
    Then there holds $r_k \to 0$ as $k\to\infty$ and 
    \begin{equation*}
        \min_{0 \leq i \leq k} r_i \leq \frac{\sqrt{\JJ(\tau_0) - \JJ_{\mathrm{opt}}}}{\sqrt{c(k+1)}} ,
    \end{equation*}
    where $\JJ_{\mathrm{opt}}$ denotes the optimal objective value and $c$ is defined by
    \begin{equation*}
        c := \gamma\,\min\{1, (1-\gamma) L_M^{-1}\}    
    \end{equation*}
    with $L_M$ being the 
    Lipschitz constant of $\nabla\JJ$ on $[-T,2T]^n$ and $\gamma$ the constant from Algorithm~\ref{alg:projgrad}.
\end{corollary}    	

\begin{proof}
    Similarly to the proof of Theorem~\ref{thm:proxconv}, we rewrite \eqref{eq:ocp} as
	\begin{equation*}
        \eqref{eq:ocp}	
        \quad \Longleftrightarrow \quad
        \min_{\tau\in \R^n} \, \widetilde\JJ(\tau) + I_{\PP}(\tau),
	\end{equation*}
    where $I_\PP$ is again the indicator functional of $\PP$ and $\widetilde\JJ: \R^n \to (\infty, \infty]$ is defined by 
    \begin{equation*}
        \widetilde\JJ(\tau) :=
        \begin{cases}
            \JJ(\tau), & - T \leq \tau_i \leq 2T \quad \forall\, i  = 1, \ldots, n.\\
            \infty, & \text{else.}
        \end{cases}
    \end{equation*}
    Then the proximal gradient method applied to this problem gives the iteration in \eqref{eq:projgrad}, 
    since $\nabla \widetilde \JJ|_{\PP} = \nabla \JJ|_{\PP}$ and all iterates are feasible by construction of the algorithm.
	Moreover, $\widetilde\JJ$ is proper and lower semicontinous, $\dom(I_\PP) \subset \interior(\dom(\widetilde\JJ))$, 
	$\dom(\widetilde\JJ)$ is convex, and $\widetilde\JJ$ is continuously differentiable 
	on $\interior(\dom(\widetilde\JJ))$ with Lipschitz continuous derivative by Theorem~\ref{thm:gradlip}.
	Therefore, \cite[Theorem~10.15]{Bec17} is applicable, which implies the result. 
\end{proof}

\begin{remark}\label{rem:assus}
    The additional assumptions in Theorem~\ref{thm:gradlip} impose severe restrictions on the data 
    of \eqref{eq:ocp}. In particular, the assumptions in \eqref{it5}--\eqref{it7} are only fulfilled by a very limited set 
    of nonlinearities. To illustrate this, let us consider a specific setting, namely 
    \begin{equation}
        \left\{\quad 
        \begin{aligned}
            \min \quad & J(y, \tau) 
            := \frac{1}{2}\, \|y - y_d\|_{L^2(0,T;L^2(\Omega))}^2 + \frac{\alpha}{2}\, \|\tau\|_{\R^n}^2 \\
            \text{s.t.} \quad & y' - \Delta y + f(y) = \sum_{i=1}^{n-1}\chi_{[\tau_i, \tau_{i+1})}\,\psi_i 
            \quad \text{in } (0,T) \times \Omega,\\
             & y(0) = y_0 \text{ in } \Omega, \quad y = 0 \text{ in } (0,T) \times \partial\Omega,\\
             & \tau \in \PP.
        \end{aligned}
        \right.
    \end{equation}
    Herein, $\Omega \subset \R^d$, $d\in \N$, is an open set, $y_d\in L^2(0,T;L^2(\Omega))$ 
    is a given desired state, and $\alpha \geq 0$ a given Tikhonov regularization parameter. 
    For the spaces $\DD$ and $X$, we choose $\DD := H^1_0(\Omega)$ and 
    $X := H^{-1}(\Omega) := H^1_0(\Omega)^*$, and $\Delta \in L(H^1_0(\Omega), H^{-1}(\Omega))$ 
    is the distributional Laplacian.
    Then Assumption~\ref{assu:spaces} is fulfilled and $-\Delta$ satisfies the assumption of maximal parabolic regularity 
	from Assumption~\ref{assu:maxparreg} by the standard $W(0,T)$-theory, see, e.g., \cite[Theorem~26.1]{Wlo87}.
    If we set the parameter $q= 2$, then 
    \begin{equation*}
        Y_q =(H^{-1}(\Omega), H^1_0(\Omega))_{1/2, 2} =L^2(\Omega)
    \end{equation*}
    and consequently, according to the assumption in \eqref{it5}, the nonlinearity has to be globally continuous as
    a mapping from $L^2(\Omega)$ to $H^{-1}(\Omega)$. Suppose that the nonlinearity is given in form of a Nemyzki
    operator $f(y)(t,x) := \mathfrak{f}(y(x,t))$ with a nonlinear function $\mathfrak{f}: \R \to \R$.
    Then this condition is only fulfilled if $\mathfrak{f}$ is globally Lipschitz continuous on the whole of $\R$. 
    A close inspection of the proof of Theorem~\ref{thm:gradlip} shows that this restrictive condition 
    is only needed to prove the existence and boundedness of solutions of the state equation. 
    If $\mathfrak{f}$ is monotone increasing, then a comparison principle with a linearized equation can be applied to 
    obtain the same result without the restrictive assumption of global Lipschitz continuity, 
    see, e.g., \cite{RZ99} for details. 
    While this assumption can thus be replaced by an alternative assumption, 
    the most severe restrictions are probably the assumptions in \eqref{it6} and \eqref{it7}. To fulfill \eqref{it7} in our setting, 
    $\mathfrak{f}$ needs to be continuously differentiable with bounded derivative such as 
    $\mathfrak{f} = \sin$ or $\mathfrak{f} = \arctan$. If this is the case, then the Nemyzki operator associated 
    with $\mathfrak{f}'$ maps $L^2(\Omega)$ to $L(L^2(\Omega), L^2(\Omega))$ and is continuous in these spaces 
    as required in \eqref{it7}.
    We underline that Assumption~\ref{assu:f} is fulfilled with such a function $\mathfrak{f}$, too. Indeed,
    if $\mathfrak{f}'$ is bounded, then $f$ is continuously 
    Fr\'echet differentiable from $H^1_0(\Omega)$ to $L^2(\Omega)$ and Assumption~\ref{assu:f} is clearly fulfilled
    with $r = \infty$. Furthermore, the boundedness of $\mathfrak{f}'$ implies that the global Lipschitz condition in 
    \eqref{it5} is satisfied. Finally, in order to satisfy the condition in \eqref{eq:fprimelip}, we additionally need to require that 
    $\mathfrak{f}'$ is globally Lipschitz continuous, too, since $q = 2$ implies that $s = 0$.
    Note that the above examples, i.e., $\mathfrak{f} = \sin$ and $\mathfrak{f} = \arctan$, satisfy all these assumptions. 
    Furthermore, the $L^2$-tracking type objective satisfies \eqref{eq:Jprimelip} and the condition in \eqref{it2}
    due to $\W^2(H^1_0(\Omega), H^{-1}(\Omega)) \embed C([0,T];L^2(\Omega))$.
    This shows that, although the assumptions in Theorem~\ref{thm:gradlip} are very restrictive, the set of problems
    satisfying them is at least not empty.
\end{remark}

	\begin{remark}
	    We emphasize that there are of course more efficient methods to solve a problem of type \eqref{eq:ocp} 
	    than the proximal gradient resp.\ projected gradient method from Algorithm~\ref{alg:projgrad}.
	    One could for instance apply Nesterov's acceleration \cite{Nes83} to obtain a
	    FISTA-type method, see, e.g., \cite{BT09}. 
	    It is however not our goal to develop an algorithm with maximum efficiency for \eqref{eq:ocp}. We rather aim 
	    to use a method as robust as possible that allows for a rigorous convergence analysis as stated in 
	    Corollary~\ref{cor:proxconv}. 
	    The reason is that we would like to demonstrate the non-convex nature of the switching-point-optimization 
	    problem, where even a robust method allowing for a convergence analysis in function space in general 
	    fails to converge to global minimizers as we will see in the next section.
	\end{remark}

	\section{Numerical experiments}\label{sec:numerics}

	For our numerical tests, we consider a setting similar to the one in Remark~\ref{rem:assus}. 
	For the spaces $X$ and $\DD$, we again choose $\DD = H^1_0(\Omega)$ and $X = H^{-1}(\Omega)$, 
	where $\Omega = (0,1)^2$ is the unit square in two spatial dimensions. 
	For the operator $A$, we choose the (negative) Laplacian $A = - \Delta : H^1_0(\Omega) \to H^{-1}(\Omega)$
	such that all conditions in Assumption~\ref{assu:spaces} and the assumption of maximal parabolic regularity
	from Assumption~\ref{assu:maxparreg} are fulfilled, see Remark~\ref{rem:assus}.
	 The nonlinearity is simply set to zero, i.e., $f\equiv 0$, such that the state equation is linear 
	 and just given by the heat equation with homogeneous Dirichlet boundary conditions. 
	 Note that $f \equiv 0$ trivially satisfies Assumption~\ref{assu:f} with $q=2$ and 
	 all additional conditions in Theorem~\ref{thm:gradlip}.
    Furthermore, the objective is given by a tracking type objective of the form 
    \begin{equation*}
        J(y, \tau) := \tfrac{1}{2} \| y(\tau) - y_d \|^2_{L^2(0,T;L^2(\Omega))}, 
    \end{equation*}
    where 
    $y_d \in L^2(0,T;L^2(\Omega))$ and $\tau_d \in \R^n$ 
    are given desired state and desired switching points.
    Note that $J$ clearly satisfies the conditions in Assumption~\ref{assu:obj}, but also the additional assumptions from 
    Theorem~\ref{thm:gradlip}.
    Finally, the form functions are of the form \eqref{eq:switch} with a function $\psi \in \DD = H^1_0(\Omega)$ 
    such that Assumption~\ref{assu:formfunc} and the additional condition in Theorem~\ref{thm:gradlip} are fulfilled, too.
    To summarize, the convergence result from Corollary~\ref{cor:proxconv} applies in this example.

    In order to be able to construct an exact globally optimal solution of \eqref{eq:ocp} for comparison, we slightly modify the state equation and consider 
    \begin{equation}\label{eq:PDEbsp}
        y'(t) - \Delta y(t) = \sum_{i=1}^{n-1}\bar\chi_{[\tau_i, \tau_{i+1})}(t)\,\psi_i + w(t), 
		\quad t \in (0,T),\quad y(0) = y_0
    \end{equation}
    with a fixed additional right hand side $g\in L^2(0,T;H^{-1}(\Omega))$. It is straight forward to see that 
    this additional affine term does not influence the above analysis. 
    The benefit is that it allows us to construct exact solutions in the following way. 
	 We arbitrarily choose a $\tau_\text{opt} \in \PP$ and a desired state 
	 $y_d \in \W^2_0(H^1_0(\Omega), H^{-1}(\Omega))$ and set 	 
	\begin{equation*}
	 	w(t,x) \coloneqq \partial_t y_d(t,x) - \Delta y_d(t,x) - \sum_{i=1}^{n-1}\bar\chi_{[\tau_i^\text{opt}, \tau_{i+1}^\text{opt})}(t)\,\psi_i(x)
	 \end{equation*}
	 to fulfill the state equation. If we then choose $\tau_d = \tau_{\mathrm{opt}}$, then it is clear that $\tau_{\mathrm{opt}}$ 
	 solves \eqref{eq:ocp} with optimal objective value $J(y_d, \tau_{\mathrm{opt}}) = 0$.
	 
	 The heat equation is discretized by means of standard continuous and piecewise linear 
	 finite elements on a Friedrich-Keller-triangulation of $\Omega$. 
	 For a given $\tau \in \R^n$, the right hand side in \eqref{eq:PDEbsp}, i.e.,
    $g := \sum_{i=1}^{n-1}\bar\chi_{[\tau_i, \tau_{i+1})}\psi_i + w$,  
	 is discretized in space by evaluating its values in the nodes $(x_i)_{i = 1, \ldots, n}$ 
	 of the finite elements mesh and afterwards interpolated using Lagrange interpolation. 
	 With regard to time, it is discretized by means of local averaging with functions 
	 $v_i \in C([0,T]; [0, \infty))$ satisfying  $\text{supp}(v_i) = [t_{i-1}, t_i]$, $ i = 1, \ldots, k$, 
	 which is due to the discontinuity of the control in time caused by the switching. 
	 Herein, $0 = t_0 < t_1 < \ldots < t_k = T$ denotes the time mesh.
	 For the temporal discretization of the PDE, we employ an implicit Euler scheme, which along 
	 with the local averaging of the right hand side leads to the following linear system of equation 
	 for each time point:
	 \begin{equation}\label{eq: rhs}
	 	M\, \frac{y_i - y_{i-1}}{\Delta_{t_i}} + A\, y_i 
	 	= \frac{\int_{t_{i-1}}^{t_i} M[g(t,x_i)]_{i = 1}^n v_i(t) \,\d t }{\int_{t_{i-1}}^{t_i} v_i(t) \,\d t},
	 	\quad i = 1, \ldots, k,
	 \end{equation}
	 where $M$ and $A$ denote mass and stiffness matrix. 
	 The objective is discretized by using Lagrange interpolation in time and space of the desired state as well as 
	 of the discretized state as solution of \eqref{eq: rhs}. 
	 Since the implicit Euler-scheme is self-adjoint, 
	 the implementation of the projected gradient method using the correct discretized gradient is straight forward. 
    
    For the construction of an exact globally optimal solution for the discretized optimization problem, we follow
    the lines of the continuous case discussed above. We first fix a set of optimal switching times and set 
    $\tau_d = \tau_{\mathrm{opt}}$. Moreover, we fix a desired state in the nodes of the Lagrange interpolation, 
    set the optimal discretized state equal to the desired one, plug that into the Euler scheme in \eqref{eq: rhs}
    and obtain the discretized counterpart of artificial right hand side $w$. The optimal objective value is thus again zero.
	 

	 \subsection{Numerical results}\label{sec:numres}
	 For our numerical examples, we set 
	 \begin{equation*}
        T \coloneqq 1, \quad y_d(t,x) \coloneqq t^2 \sin(\pi x_1) \sin(\pi x_2), \quad y_0 \coloneqq 0, \quad 
        \psi(x) \coloneqq  10 \exp(-2 \| x - \tfrac 12 \|^2_{\R^n}),
	 \end{equation*}
	 while the optimal switching times are set to 
    \begin{equation}\label{eq:tauopt}
		\tau_{\mathrm{opt}} = (0,  0.1, 0.15, 0.25, 0.60, 0.65, 0.70, 0.75, 0.80, 0.85).
    \end{equation}
    We emphasize that this is a generic choice of $\tau_{\mathrm{opt}}$ without any specific properties.
	 The local averaging functions in \eqref{eq: rhs} are chosen to be continuous and 
	  piecewise linear on the temporal mesh and satisfy
	  \begin{equation*}
	  	 v_i(t_{i-1}) = 0, \quad v_i(\tfrac{t_{i-1} + t_i}{2}) = 1, \quad v_i(t_i) = 0, \quad i = 1, \ldots, k.
	  \end{equation*}
	  This allows us to compute the integrals in \eqref{eq: rhs} exactly. 
	  In the following numerical experiments, the time interval $[0,1]$ is discretized by an equidistant mesh with 1001 points, 
	  while the spatial domain is discretized by a uniform Friedrich-Keller triangulation with 40 points in each spatial dimension. 
	  The algorithm terminates if any of the following stopping criteria is satisfied: 
	  \begin{itemize}
	  	\item the maximum number of 200 iterations is reached, 
	  	\item the norm of the residuum in the fixed point equation is smaller than a tolerance of $10^{-8}$, i.e., 
	  	\begin{equation*}
	  		r_k = \| \Pi_{\PP} (\tau_k - L_k^{-1}\nabla \JJ(\tau_k)) - \tau_k \|_{\R^n} < 10^{-8},
	  	\end{equation*}
	  	\item the relative change in the switching times is smaller than a tolerance of $10^{-8}$, i.e., 
	  	\begin{equation*}
	  		\frac{\| \tau_k - \tau_{k+1}\|_{\R^n}}{\|\tau_k \|_{\R^n}} < 10^{-8}. 
	  	\end{equation*}
	  \end{itemize}
	We present two types of numerical experiments, differing in the choice of the exact solution
	\begin{itemize}
		\item Case (i): exact solution in function space,
		\item Case (ii): exact solution after discretization. 
	\end{itemize}
	In the following, we present the results of 10 numerical tests with different initial values $\tau_0$, 
	which are randomly chosen without special properties.
	
\begin{table}[!h]
	\centering
\resizebox{\linewidth}{!}{	
\begin{tabular}{c c c @{\hspace{5pt}} c}
		\hline
		iteration & objective value & $r_k$ & $\tau_k$ \\
		\hline
		10 & 1.8622e-03 &   6.1176e-03 &  (0.000 0.100 0.150 0.250 0.481 0.484 0.697 0.830 0.995 1.000)\\
		100 & 1.8613e-03 &   4.6699e-07 &  (0.000 0.100 0.150 0.250 0.482 0.486 0.697 0.830 0.997 1.000)\\
		200 & 1.8613e-03 &   7.2219e-08 &  (0.000 0.100 0.150 0.250 0.482 0.486 0.697 0.830 0.997 1.000)\\
		\hline
	\end{tabular}}
	\caption{Case (i), $\tau_0 = (0.05, 0.1, 0.15, 0.25, 0.4, 0.55, 0.65, 0.85, 0.9, 1)$ }
	\label{tab:example1}
\end{table}
\begin{table}[!h]
	\centering
\resizebox{\linewidth}{!}{	
\begin{tabular}{c c c @{\hspace{5pt}} c}
	\hline
	iteration & objective value & $r_k$ & $\tau_k$ \\
		\hline
		10 & 1.7542e-03 &   3.7323e-03 &  (0.000 0.100 0.150 0.250 0.549 0.567 0.698 0.827 0.907 0.910)\\
		100 & 8.3545e-04 & 9.4382e-08 &  (0.000 0.100 0.150 0.250 0.600 0.652 0.719 0.832 0.909 0.911)\\
		108 & 8.3545e-04 &   7.2652e-08 &  (0.000 0.100 0.150 0.250 0.600 0.652 0.719 0.832 0.909 0.911)\\
		\hline
	\end{tabular}}
	\caption{Case (i), $\tau_0 = (0, 0.1, 0.15, 0.2, 0.55, 0.55, 0.65, 0.85, 0.9, 0.9)$ }
	\label{tab:example2}
\end{table}
\begin{table}[!h]
	\centering
\resizebox{\linewidth}{!}{	
\begin{tabular}{c c c @{\hspace{5pt}} c}
	\hline
	iteration & objective value & $r_k$ & $\tau_k$ \\
		\hline
	10 & 1.5858e-06 &   1.0682e-03 &  (0.000 0.100 0.150 0.250 0.600 0.649 0.699 0.749 0.799 0.850)\\
	100 & 7.3990e-08 & 4.4680e-07 &  (0.000 0.100 0.150 0.250 0.600 0.650 0.700 0.750 0.800 0.850)\\
	181 & 7.3989e-08  & 6.4023e-08 &  (0.000 0.100 0.150 0.250 0.600 0.650 0.700 0.750 0.800 0.850)\\
		\hline
	\end{tabular}}
	\caption{Case (i), $\tau_0 = (0.05, 0.15, 0.15, 0.25, 0.55, 0.6, 0.7, 0.75, 0.8, 0.9)$ }
	\label{tab:example3}
\end{table}
\begin{table}[!h]
	\centering
\resizebox{\linewidth}{!}{	
\begin{tabular}{c c c @{\hspace{5pt}} c}
	\hline
	iteration & objective value & $r_k$ & $\tau_k$ \\
		\hline
		10 & 1.0351e-03 &  1.4137e-02 &  (0.001 0.105 0.149 0.251 0.596 0.641 0.691 0.769 0.991 1.000)\\
		100 & 9.4102e-04 &  1.8979e-05 &  (0.000 0.100 0.150 0.250 0.600 0.650 0.704 0.776 0.995 1.000)\\
		175 & 9.4102e-04 &  2.6253e-08 &  (0.000 0.100 0.150 0.250 0.600 0.650 0.704 0.776 0.995 1.000)\\
		\hline
	\end{tabular}}
	\caption{Case (i), $\tau_0 = (0.05, 0.15, 0.25, 0.3, 0.45, 0.6, 0.7, 0.75, 0.9, 1)$ }
	\label{tab:example4}
\end{table}
\begin{table}[!h]
	\centering
\resizebox{\linewidth}{!}{	
\begin{tabular}{c c c @{\hspace{5pt}} c}
	\hline
	iteration & objective value & $r_k$ & $\tau_k$ \\
		\hline
10 & 1.8478e-03 &  1.2423e-02 &  (0.005 0.236 0.257 0.260 0.600 0.650 0.702 0.769 0.870 0.891)\\
100 & 1.0447e-03 &  1.0978e-05 &  (0.005 0.243 0.258 0.259 0.600 0.650 0.700 0.750 0.800 0.850)\\
200 & 1.0447e-03 &  8.0109e-07 &  (0.005 0.243 0.258 0.259 0.600 0.650 0.700 0.750 0.800 0.850)\\
		\hline
	\end{tabular}}
	\caption{Case (i), $\tau_0 = (0.05, 0.15, 0.25, 0.45, 0.6, 0.65, 0.7, 0.75, 0.9, 0.9)$ }
	\label{tab:example5}
\end{table}
\begin{table}[!h]
	\centering
\resizebox{\linewidth}{!}{	
\begin{tabular}{c c c @{\hspace{5pt}} c}
	\hline
	iteration & objective value & $r_k$ & $\tau_k$ \\
		\hline
	10 & 1.8605e-03 &   1.3095e-03 &  (0.000 0.100 0.150 0.250 0.485 0.488 0.698 0.827 0.996 1.000)\\
	50 & 1.8588e-03 &   7.4467e-04 &  (0.000 0.100 0.150 0.250 0.486 0.489 0.698 0.830 0.997 1.000)\\
	76 & 1.8586e-03 &   6.5060e-08 &  (0.000 0.100 0.150 0.250 0.487 0.490 0.698 0.830 0.997 1.000)\\
	\hline
	\end{tabular}}
	\caption{Case (ii), $\tau_0 = (0.05, 0.1, 0.15, 0.25, 0.4, 0.55, 0.65, 0.85, 0.9, 1)$ }
	\label{tab:example6}
\end{table}
\begin{table}[!h]
	\centering
\resizebox{\linewidth}{!}{	
\begin{tabular}{c c c @{\hspace{5pt}} c}
	\hline
	iteration & objective value & $r_k$ & $\tau_k$ \\
		\hline
		10 & 1.5627e-03 &  6.0360e-03 &  (0.000 0.100 0.149 0.250 0.569 0.594 0.700 0.829 0.902 0.904)\\
		100 & 8.4003e-04 &   6.4796e-08 &  (0.000 0.100 0.150 0.250 0.600 0.652 0.719 0.829 0.901 0.903)\\
		200 & 8.4003e-04 &   6.2153e-06 &  (0.000 0.100 0.150 0.250 0.600 0.652 0.719 0.829 0.901 0.903)\\
		\hline
	\end{tabular}}
	\caption{Case (ii), $\tau_0 = (0, 0.1, 0.15, 0.2, 0.55, 0.55, 0.65, 0.85, 0.9, 0.9)$ }
	\label{tab:example7}
\end{table}
\begin{table}[!h]
	\centering
\resizebox{\linewidth}{!}{	
\begin{tabular}{c c c @{\hspace{5pt}} c}
	\hline
	iteration & objective value & $r_k$ & $\tau_k$ \\
		\hline
	10 & 1.4696e-06 &   2.6819e-03 &  (0.000 0.100 0.150 0.250 0.600 0.650 0.700 0.751 0.799 0.850)\\
	100 & 1.8157e-08 &   3.9587e-07 &  (0.000 0.100 0.150 0.250 0.600 0.650 0.700 0.750 0.800 0.850)\\
	200 & 1.8154e-08 &   4.4084e-07 &  (0.000 0.100 0.150 0.250 0.600 0.650 0.700 0.750 0.800 0.850)\\
		\hline
	\end{tabular}}
	\caption{Case (ii), $\tau_0 = (0.05, 0.15, 0.15, 0.25, 0.55, 0.6, 0.7, 0.75, 0.8, 0.9)$ }
	\label{tab:example8}
\end{table}
\begin{table}[!h]
	\centering
\resizebox{\linewidth}{!}{	
\begin{tabular}{c c c @{\hspace{5pt}} c}
	\hline
	iteration & objective value & $r_k$ & $\tau_k$ \\
		\hline
		10 & 1.5788e-03 &  2.5297e-02 &  (0.003 0.141 0.167 0.252 0.587 0.639 0.692 0.770 0.994 1.000)\\
		100 & 9.3239e-04 &  2.4008e-07 &  (0.000 0.100 0.150 0.250 0.600 0.650 0.704 0.775 0.995 1.000)\\
		172 & 9.3239e-04 &  3.8270e-08 &  (0.000 0.100 0.150 0.250 0.600 0.650 0.704 0.775 0.995 1.000)\\
		\hline
	\end{tabular}}
	\caption{Case (ii), $\tau_0 = (0.05, 0.15, 0.25, 0.3, 0.45, 0.6, 0.7, 0.75, 0.9, 1)$ }
	\label{tab:example9}
\end{table}
\begin{table}[!h]
	\centering
\resizebox{\linewidth}{!}{	
\begin{tabular}{c c c @{\hspace{5pt}} c}
	\hline
	iteration & objective value & $r_k$ & $\tau_k$ \\
		\hline
10 & 1.8230e-03 &  6.4562e-03 &  (0.000 0.244 0.284 0.284 0.601 0.650 0.703 0.770 0.869 0.887)\\
50 & 1.0587e-03 &  2.9191e-05 &  (0.000 0.244 0.285 0.286 0.600 0.650 0.700 0.750 0.800 0.850)\\
93 & 1.0587e-03 &  3.4054e-08 &  (0.000 0.244 0.285 0.286 0.600 0.650 0.700 0.750 0.800 0.850)\\
	\hline
	\end{tabular}}
	\caption{Case (ii), $\tau_0 = (0.05, 0.15, 0.25, 0.45, 0.6, 0.65, 0.7, 0.75, 0.9, 0.9)$ }
	\label{tab:example10}
\end{table}

We observe that, except for the cases in Tables~\ref{tab:example3} and \ref{tab:example8}, 
the algorithm does not seem to converge to the global 
minimizer $\tau_{\mathrm{opt}}$. 
The iteration frequently stops in a stationary point, where the residual of the fixed point equation becomes comparatively small, 
but the limit substantially differs from $\tau_{\mathrm{opt}}$.
This underlines the non-convex nature of the switching-point-to-control map, which is the only source of non-convexity 
in this example. 
Note that the value of the objective is also significantly larger in stationary points different from $\tau_{\mathrm{opt}}$
compared to the examples in Table~\ref{tab:example3} and \ref{tab:example8}, 
where $\tau_k \approx \tau_{\mathrm{opt}}$.

To summarize, we observe that we frequently do not converge to the global minimizer, 
even if the initial value is not that far away from the optimal value $\tau_{\mathrm{opt}}$.
This underlines the non-convex nature of the problem under consideration.

\section{Convexification of the Feasible Set}\label{sec:convex}

As the numerical results from Section~\ref{sec:numres} demonstrate, we can in general not expect 
that a standard algorithm from nonlinear optimization delivers a global minimizer of \eqref{eq:ocp}, 
even in the probably most simple case, where the parabolic equation is linear and the objective is quadratic.
Therefore, if one aims to solve the problem globally, then one has to resort to other strategies such as 
branch-and-bound. In this context, the computation of dual bounds is probably the most challenging task.
This requires to construct convex relaxations, where stronger bounds are obtained by tighter relaxations. In particular, we are interested in the convex hulls of our feasible sets, which yield best possible convexifications (as long as the objective function is not taken into account).

The only source of non-convexity in the numerical example in Section~\ref{sec:numerics} is the 
switching-point-to-control operator $B$. To obtain a tight convexification, let us define
the \emph{set of feasible switching functions} by 
\begin{equation*}
\begin{aligned}
	\FF = \big\{ (\tau, \chi) = (\tau_1, \ldots, \tau_n, \chi_1, \ldots, \chi_{n-1}) \in \R^n \times L^q(0,T)^{n-1} \colon \quad &\\
    \tau \in \PP, \quad \chi_i = \chi_{[\tau_i, \tau_{i+1})},  \; i = 1, \ldots, n-1 & \big\}. 
\end{aligned}
\end{equation*}
Moreover, we define the following linear and continuous operator
\begin{equation*}
    L: L^q(0,T)^{n-1} \to L^q(0,T;X) , \quad
    L \chi := \sum_{i=1}^n \chi_i \psi_i,
\end{equation*}
where $\psi_1, \ldots, \psi_{n-1}$ are the form functions from Assumption~\ref{assu:formfunc}.
Then, we can rewrite \eqref{eq:ocp} as
\begin{equation*}
    \eqref{eq:ocp}
    \Longleftrightarrow
    \left\{ 
    \begin{aligned}
        ~ \min \quad & J(y,\tau) \\
		\text{s.t.} \quad & \text{$y\in L^q(0,T;\DD)$ is a weak solution associated with $u = L \chi$},\\
        & (\tau, \chi) \in \FF. 
    \end{aligned}
    \right.
\end{equation*}
Thus, if $J$ is quadratic and the parabolic equation is linear, the only non-convex part of \eqref{eq:ocp} is hidden 
in the set $\FF$ of feasible switching functions and a tight convexification of \eqref{eq:ocp} reads 
\begin{equation}\label{eq:convocp}
    \left\{\quad 
    \begin{aligned}
        \min \quad & J(y,\tau) \\
		\text{s.t.} \quad & \text{$y\in L^q(0,T;\DD)$ is a weak solution associated with $u = L \chi$},\\
        & (\tau, \chi) \in \overline{\conv(\FF)}^{\R^n \times L^q(0,T)^{n+1}}. 
    \end{aligned}
    \right.
\end{equation}
The aim of this section is to give a precise characterization of the closure of the convex hull of $\FF$.
As we will see, this is achieved by the following set
\begin{equation*}
\begin{aligned}
    \FF^{\mathrm{ext}} \coloneqq 
    \Big\{ (\tau, \chi, z) &= (\tau_1, \ldots, \tau_n, \chi_1, \ldots, \chi_{n-1}, z_1, \ldots, z_n)  \\[-1ex]
    &\in \R^n \times L^q(0,T)^{n-1}  \times L^q(0,T)^n: \\ 
    & \quad
    \begin{aligned}
        & \chi_i, z_j \in \BV(0,T; [0,1]) & & \forall\,  i = 1, \ldots, n-1, \; j = 1, \ldots, n,  \\
        & \chi_i(t) = z_i(t) - z_{i+1}(t) & & \forall \, i = 1, \ldots, n-1,  \text{ f.a.a.\ }t \in (0,T),\\
        & z_{i}(t) \leq z_{i-1}(t) & &  \forall\, i = 2, \ldots, n, \text{ f.a.a.\ }t \in (0,T), \\
        & Dz_{i} \geq 0 & & \forall \, i = 1, \ldots, n,   \\ 
        & \tau_i = \int_0^T (1- z_i(t)) \,\d t & & \forall\,  i = 1, \ldots, n  \Big\},
    \end{aligned}
\end{aligned}
\end{equation*}
where $Dz_i \geq 0$ means that $Dz_i$ is a non-negative measure, i.e., $Dz_i(A) \geq 0$ for all $A \in \mathcal B(0,T)$. 

Given a function $v\in \BV(0,T)$, we associate a particular representative with the equivalence class $[v]$, 
namely the function $V$ defined by 
\begin{equation}\label{eq:goodrepresentative}
	V(t) := w(t) + \frac{1}{T} \int_0^T (v(s) - w(s)) \, \d s,
	\quad
	w(t) := 
	\begin{cases}
		Dv((0,t]), & t \geq 0, \\
		Dv((0,T)), & t = T.
	\end{cases}
\end{equation}
We call the function from \eqref{eq:goodrepresentative} the \emph{good representative} of $[v]$. 
It is easily seen that $V$ is continuous from the right in every point in $[0,T)$ and continuous from the left in $t = T$. 

\begin{lemma}\label{lem:closed, conv}
	The set $\FF^{\mathrm{ext}}$ is nonempty, convex, and closed. 
\end{lemma}

\begin{proof}
    Since all constraints in $\FF^{\mathrm{ext}}$ are affine, the set is obviously convex.
    Moreover, due to $(\boldsymbol{1},0, 0) \in \FF^{\mathrm{ext}}$ with $\boldsymbol{1} := (1, \ldots, 1) \in \R^n$,
    it is nonempty.
    
	To show its closedness, let $\{(\tau^{(n)}, \chi^{(n)}, z^{(n)})\}_{n \in \N} \subset \mathcal F^{\text{ext}}$ 
	be a sequence converging to some $(\tau, \chi, z)$, i.e., 
    \begin{equation*}
	    \tau^{(n)} \to \tau  \; \text{ in } \R^n,\quad
	    \chi^{(n)} \to \chi  \;  \text{ in } L^q(0,T)^{n-1}, \quad  
	    z^{(n)}_i \to z_i  \; \text{ in } L^q(0,T)^n, 
    \end{equation*}
    as $n \to \infty$. From strong convergence in $L^q(0,T)$, it immediately follows that 
	\begin{equation*}
		\chi_i = z_i - z_{i+1}, \quad i = 1, \ldots, n-1, 
		\quad 
		z_j \leq z_{j-1}, \quad j = 2, \ldots, n, 
		\quad \text{a.e.\ in } (0,T)
	\end{equation*}
	as well as $0 \leq \chi_i, z_j \leq 1$, $i = 1, \ldots, n-1$, $j = 1, \ldots, n$,  a.e.\ in $(0,T)$ and 
	\begin{equation*}
	 	\tau_i = \lim\limits_{n \to \infty} \tau_i^{(n)} = \lim\limits_{n \to \infty} \int_0^T (1-z_i^{(n)}(t))  \,\d t =  \int_0^T (1-z_i(t))  \,\d t. 
	\end{equation*}
	Due to $Dz^{(n)}_i \geq 0$ and $Z^{(n)}_i(t) \in [0,1]$
        for all $t \in [0,T], i = 1, \ldots, n$, by the 
	one-sided continuity of the good representative, we moreover obtain 
	\begin{equation*}
		|z_i^{(n)}|_{\TV} = \int_{0}^{T} \d Dz_i^{(n)}(t) = Z_i^{(n)}(T) - Z_i^{(n)}(0) \leq 1, 
	\end{equation*}
	and, since $\chi_i = z_i - z_{i+1}$, it holds that 
	\begin{equation*}
			|\chi_i^{(n)}|_{\TV} \leq 	|z_i^{(n)}|_{\TV} + 	|z_{(i+1)}^{(n)}|_{\TV} \leq 2. 
	\end{equation*}
	Consequently, the sequences $\{\chi^{(n)}\}_{n \in \N}$, $\{z^{(n)}\}_{n \in \N}$
    are bounded in $\BV(0,T)^{n-1}$ and $\BV(0,T)^{n}$, respectively. 
    Owing to \cite[Proposition~10.1.1]{ABM06}, we can thus select
    subsequences converging weakly in $\BV$, such that the limits satisfy 
    $\chi_i, z_j \in \text{BV}(0,T), i = 1, \ldots, n-1, j = 1, \ldots, n$. 
    Moreover, the set $\{ \mu \in \mathfrak M((0,T)) \colon \mu \geq 0\}$
    is easily seen to be closed w.r.t.\ weak-* convergence in $\mathfrak M((0,T))$.
    Hence, the weak convergence of a subsequence of $(z_i^{(n)})_{n \in \N}$ to $z_i$ in $\text{BV}(0,T)$ 
    implies $Dz_i \geq 0$, $i=1, \ldots, n$. We have thus shown $(\tau, \chi, z) \in \FF^{\mathrm{ext}}$. 	
\end{proof}

\begin{lemma}\label{lem:densityFF}
    The set
    	\begin{equation*}
    	\begin{aligned}
		\FF^{\mathrm{ext}}_h \coloneqq 
		\big\{  & (\tau, \chi, z) \in \FF^{\mathrm{ext}} \colon \\
		 & \quad \chi_1, \ldots, \chi_{n-1}, z_1, \ldots, z_n \text{ are piecewise constant a.e.\ in } (0,T) \big\}
    	\end{aligned}
	\end{equation*}
	is dense in $\FF^{\mathrm{ext}}$.
\end{lemma}

\begin{proof}
    Let $(\tau, \chi, z) \in \FF^{\mathrm{ext}}$ be arbitrary. For $N\in \N$, we consider the equidistant 
    partition 
    \begin{equation*}
        0 = t_0 < t_1 := \tfrac{T}{N} < t_2 := 2\, \tfrac{T}{N} < \ldots < t_{N-1} := (N-1) \tfrac{T}{N} < t_N = T
    \end{equation*}
    and define the piecewise constant approximations 
    \begin{equation*}
        \chi_i^{(N)} := \sum_{k=1}^N \frac{N}{T} \int_{t_{k-1}}^{t_k} \chi_i(t)\,\d t \,\chi_{(t_{k-1}, t_k]}, \quad
        z_j^{(N)} := \sum_{k=1}^N \frac{N}{T} \int_{t_{k-1}}^{t_k} z_j(t)\,\d t \,\chi_{(t_{k-1}, t_k]}
    \end{equation*}
    for $i = 1, \ldots, n-1$ and $j = 1, \ldots, n$. Since $\chi \in \BV(0,T)^{n-1}$ and $z\in \BV(0,T)^n$, 
    $\chi^{(N)}$ and $z^{(N)}$ converge strongly in $L^q(0,T)^{n-1}$ and $L^q(0,T)^n$, respectively, 
    to $\chi$ and $z$ as $N\to \infty$, cf., e.g., \cite[Lemma~2.5]{Buc25}.
    This implies that 
    \begin{equation*}
        \tau^{(N)} := \Big( \int_0^T \big(1 - z_i^{(N)}(t) \big) \d t\Big)_{i=1}^n
    \end{equation*}
    converges to $\tau$.
    It remains to show that $(\tau^{(N)}, \chi^{(N)}, z^{(N)}) \in \FF^{\mathrm{ext}}_h$.
    It is clear that the piecewise means and thus $\chi^{(N)}$ and $z^{(N)}$ satisfy the first to inequality 
    constraints in the definition of $\FF^{\mathrm{ext}}$, i.e., 
    $\chi^{(N)}_i = z_i^{(N)} - z_{i+1}^{(N)}$ and $z_{i}^{(N)} \leq z_{i-1}^{(N)}$ a.e.\ in $(0,T)$. 
    The non-negativity of $D z_i^{(N)}$ finally follows from \cite[Lemma~2.6]{Buc25}. 
\end{proof}

We are now in the position to state the relation between $\FF^{\mathrm{ext}}$ and the convex hull of $\FF$.

\begin{theorem}
    There holds that 
    \begin{equation*}
        \pi_{(\tau, \chi)} (\FF^{\mathrm{ext}}) = \overline{\conv(\FF)}^{\R^n \times L^q(0,T)^{n-1}},     
    \end{equation*}
    where $\pi_{(\tau, \chi)}$ denotes the projection onto the first two components of $\FF^{\mathrm{ext}}$, i.e., 
    \begin{equation*}
        \pi_{(\tau, \chi)} (\FF^{\mathrm{ext}}) := 
        \{ (\tau, \chi) \in \R^n \times L^q(0,T)^{n-1} | \;
        \exists\, z \in L^q(0,T)^n \colon (\tau, \chi, z) \in \FF^{\mathrm{ext}} \} .
    \end{equation*}     
\end{theorem}

\begin{proof}
	To prove ``$\subset$'', we first consider an arbitrary element  $(\tau, \chi,  z) \in \FF^{\mathrm{ext}}_h$ and denote 
	the good representatives of $z = (z_1, \ldots, z_n)$ by $Z_1, \ldots, Z_n$. Since these are 
	continuous from the right in $[0,T)$, there exists a partition $0 = t_0 < t_1 < t_2 < \ldots < t_N \leq T$ 
	of the time horizon $[0,T]$ such that $Z_1, \ldots, Z_n$ jump at most in the grid points, i.e., 
	$Z_1, \ldots, Z_n$ are constant in $(t_{k-1}, t_{k})$ for all $k = 1, \ldots, N$. Let further 
	\begin{equation}\label{eq:zvalues}
		0 \leq v_1 \leq \ldots \leq v_m \leq 1, \quad (m\in \N)
	\end{equation}
	be the function values of $Z_1, \ldots, Z_n$ in ascending order and set $v_0 := 0, v_{m+1} := 1$. 
    Note that $v_m \leq 1$ due to $z \in \BV(0,T;[0,1])$. Note moreover that, by construction, 
    \begin{equation}\label{eq:convexcomb}
        \sum_{i = 1}^{m+1} (v_i - v_{i-1}) = v_{m+1} - v_0 = 1.
    \end{equation}    	
	Now, for each $i = 1, \ldots, m+1, \, j = 1, \ldots, n$, we set 
	\begin{align*}
		\tau_{j}^{(i)} &:= \inf \{ t \in \{ t_0, \ldots, t_N \}: Z_j(t) \geq v_i\}, 
	\end{align*}
	with the convention $\inf \emptyset := T$.  
	Because of $Z_1 \geq Z_2 \geq \ldots \geq Z_n$, we then have 
	\begin{equation}\label{eq:tauiinP}
	    \tau_{1}^{(i)} \leq \tau_{2}^{(i)} \leq \ldots \leq \tau_{n}^{(i)} \quad \forall\, i = 1, \ldots, m+1
	\end{equation}
    and the construction of $\tau_{j}^{(i)}$ implies 
    \begin{equation*}
	\begin{aligned}
		z_j & = \sum_{i = 1}^{m+1} (v_i - v_{i-1}) \chi_{[\tau_{j}^{(i)}, T]}, & & j = 1, \ldots, n, \\ 
		\chi_j &=  z_j - z_{j+1}
		=  \sum_{i = 1}^{m+1} (v_i - v_{i-1})  \chi_{[\tau_{j}^{(i)}, \tau_{j+1}^{(i)})}, & & j = 1, \ldots, n-1,
	\end{aligned}    
    \end{equation*}
	a.e.\ in $(0,T)$. 
	From that and \eqref{eq:convexcomb}, it follows that 
    \begin{equation*}
        \tau_j = \int_0^T (1 - z_j)\d t 
        = \sum_{i=1}^{m+1} (v_i - v_{i-1}) \int_0^T (1 - \chi_{[\tau_{j}^{(i)}, T]}) \d t
        = \sum_{i=1}^{m+1} (v_i - v_{i-1})  \tau_{j}^{(i)}
    \end{equation*}    	
	Therefore, if we set 
    \begin{equation*}
        \tau^{(i)} := \big( \tau_{1}^{(i)}, \ldots,  \tau_{n}^{(i)} \big), \quad
        \chi^{(i)} := \big( \chi_{[\tau_{1}^{(i)}, \tau_{2}^{(i)})}, \ldots, \chi_{[\tau_{n-1}^{(i)}, \tau_{n}^{(i)})} \big),
    \end{equation*}
	then \eqref{eq:tauiinP} implies $(\tau^{(i)}, \chi^{(i)}) \in \FF$ for every $i = 1, \ldots, m+1$ 
	and we obtain
	\begin{equation*}
	    \tau = \sum_{i=1}^{m+1} (v_i - v_{i-1})\tau^{(i)}
	    \quad \text{and} \quad 
	    \chi = \sum_{i=1}^{m+1} (v_i - v_{i-1})\chi^{(i)}.
	\end{equation*}
    From \eqref{eq:zvalues} and \eqref{eq:convexcomb} it follows that $(\tau, \chi) \in \conv(\FF)$.
    Since $(\tau, \chi, z) \in \FF^{\mathrm{ext}}_h$ was arbitrary, this implies that 
    $\pi_{(\tau, \chi)} (\FF^{\mathrm{ext}}_h )  \subset \conv(\FF)$ and the density of 
    $\FF^{\mathrm{ext}}_h$ in $\FF^{\mathrm{ext}}$ by Lemma~\ref{lem:densityFF} yields     
    $\pi_{(\tau, \chi)} (\FF^{\mathrm{ext}}) \subset \overline{\conv(\FF)}$. 

    For the reverse inclusion, we first note that $\FF \subset \pi_{(\tau, \chi)} (\FF^{\mathrm{ext}})$, since, 
    given $(\tau_1, \ldots, \tau_n, \chi) \in \FF$, we can simply set $z_j := \chi_{[\tau_j, T)} \in  \BV(0,T;[0,1])$, $j = 1, \ldots, n$, 
    and observe that the conditions in $\FF^{\mathrm{ext}}$ are fulfilled with this choice of $z$ due to 
    $\chi_i = \chi_{[\tau_i, \tau_{i+1})}$ by the definition of $\FF$.
    Since $\FF^{\mathrm{ext}}$ is a convex set, its projection $\pi_{(\tau, \chi)} (\FF^{\mathrm{ext}})$ is convex as well. 
    To prove the inclusion, it thus remains to show that $\pi_{(\tau, \chi)} (\FF^{\mathrm{ext}})$ is closed. 
    For this purpose, let $\{(\tau^{(n)}, \chi^{(n)})\}_{n \in \N} \subset \pi_{(\tau, \chi)} (\FF^{\mathrm{ext}})$ 
    be a sequence converging to some $(\tau, \chi) \in \R^n \times L^q(0,T)^{n-1}$. 
    By assumption, there exists a sequence $\{z^{(n)}\}_{n \in \N} \subset L^q(0,T)^n$ such that 
    $(\tau^{(n)}, \chi^{(n)}, z^{(n)})_{n \in \N} \subset \FF^{\mathrm{ext}}$ for all $n\in \N$. 
    Due to $z^{(n)} \in [0,1]^n$ a.e.\ in $(0,T)$, 
    there exists a weakly converging subsequence $\{z^{(n_k)}\}_{k\in \N}$, i.e., $z^{(n_k)} \weak z^*$ in 
    $L^q(0,T)^n$ as $k\to \infty$. Consequently, the sequence $\{(\tau^{(n_k)}, \chi^{(n_k)}, z^{(n_k)})\}_{k\in \N}$ 
    converges weakly to $(\tau, \chi, z^*)$. Since $\FF^{\mathrm{ext}}$ is closed and convex 
    by Lemma~\ref{lem:closed, conv}, it is weakly sequentially closed such that $(\tau, \chi, z^*) \in \FF^{\mathrm{ext}}$ 
    and therefore $(\tau, \chi) \in \pi_{(\tau, \chi)} (\FF^{\mathrm{ext}})$ 
    which shows the closedness of $\pi_{(\tau, \chi)} (\FF^{\mathrm{ext}})$.
\end{proof}

The set $\FF^{\mathrm{ext}}$ can be interpreted as an \emph{extended  formulation}, i.e., 
a characterization of the convex hull of the feasible set in terms of (affine-)linear equations
and inequalities at the price of introducing additional variables (in our case the variable~$z$). 
In finite dimensional optimization and, in particular, in combinatorial optimization, 
extended formulations are a well established tool for the computation of dual bounds or even exact solutions, we
only refer to \cite{CCZ13} and the references therein. Concerning infinite dimensional problems, 
the research on this topic has just been initiated in \cite{Buc25}. 

A particularly interesting fact about our extended formulation is that the switching times 
depend linearly on the new variable $z$ and are therefore superfluous. In other words, 
the convexified problem  \eqref{eq:convocp} can be formulated in terms of the new variable $z$ only
by using that 
\begin{equation*}
    \chi_i = z_i - z_{i-1}
    \quad \text{and} \quad 
    \tau_i = \int_0^T (1 - z_i) \d t.
\end{equation*}
This raises the question whether the switching points $\tau_1, \ldots, \tau_n$ are useful 
optimization variables when it comes to convexification.
Our findings indicate that the switching point approach might not be suitable 
for global optimization of switching processes.

We expect that the extended formulation $\FF^{\mathrm{ext}}$ can be used for the construction 
of a branch-and-bound algorithm following the lines of \cite{BGM25}. 
This however goes beyond the scope of this work and is subject to future research.


\section*{Acknowledgment}

We are very grateful to Daniel Wachsmuth (U W\"urzburg) for pointing out reference \cite{KM22} to us.

\begin{appendix}

\section{Autonomous parabolic equations}

\begin{lemma}\label{lem:normGt}
    Let $t > 0$ be given. For every $u \in L^q(0,t;X)$ and every $y_0 \in Y_q$, the equation
    \begin{equation*}
        y' + A y = u \quad \text{in } (0,t), \quad y(0) = y_0.
    \end{equation*}
    admits a unique solution 
    \begin{equation*}
        y \in \W^q(0,t;\DD, X) \embed C([0,t];Y_q).
    \end{equation*}
    The associated solution operator 
    $G_t : L^q(0,t;X) \times Y_q \to C([0,t];Y_q)$ is linear and bounded. 
    Moreover, there is a constant $C_T$ such that 
    \begin{equation*}
        \sup_{t\in [0,T]} \|G_t\|_{L(L^q(0,t;X) \times Y_q, C([0,t];Y_q))} \leq C_T.
    \end{equation*}     
\end{lemma}

\begin{proof}
    The existence and uniqueness of the solution is guaranteed by maximal parabolic regularity of $A$. 
    Using the linearity of the equation, we split $G_t$ into two parts, namely 
    \begin{equation*}
        G_t := K_t + \iota_t, 
    \end{equation*}
    where $K_t:  L^q(0,t;X) \to C([0,t];Y_q)$ is the solution operator to 
    \begin{equation*}
        y' + A y = u \quad \text{in } (0,t), \quad y(0) = 0, 
    \end{equation*}
    while $\iota_t: Y_q \to C([0,t];Y_q)$ is the solution operator to 
    \begin{equation*}
        y' + A y = 0 \quad \text{in } (0,t), \quad y(0) = y_0.
    \end{equation*}
    By applying translation and extension by zero, it is shown in \cite[Lemma~3.16]{HMRR09} 
    that the norm of $K_t$ does not increase if $t$ is decreased and thus 
    \begin{equation*}
        \sup_{t\in [0,T]} \|K_t\|_{L(L^q(0,t;X), C([0,t];Y_q))}
        \leq \|K_T\|_{L(L^q(0,t;X), C([0,t];Y_q))}.
    \end{equation*}     
    Concerning $\iota_t$, we first observe that
    \begin{equation*}
        \|e^{-tA}\|_{L(\DD, \DD)} \leq \|e^{-tA}\|_{L(X,X)},
    \end{equation*}
    since
    \begin{equation*}
    \begin{aligned}
        \|e^{-tA} x \|_{\DD}
        &= \|A e^{-tA} x\|_X\\
        &= \|e^{-tA} A x\|_X
        \leq \|e^{-tA} \|_{L(X,X)} \|A x\|_X
        = \|e^{-tA} \|_{L(X,X)} \|x\|_\DD
        \quad \forall\, x\in \DD.    
    \end{aligned}
    \end{equation*}        
    Via interpolation, this implies 
    \begin{equation*}
        \|e^{-tA}\|_{L(Y_q, Y_q)} \leq \|e^{-tA}\|_{L(X,X)} \leq M\, e^{\omega t},
    \end{equation*}
    with constants $M>0$ and $\omega \in \R$, where the last inequality is a consequence of $A$
    being the generator of an analytic semigroup. Thus we obtain
    \begin{equation*}
    \begin{aligned}
        \sup_{t\in [0,T]} \|\iota_t\|_{L(Y_q, C([0,t];Y_q))}
        = \sup_{t\in [0,T]} \max_{s\in [0,t]} \|e^{-sA}\|_{L(Y_q, Y_q)} \leq \max\{M, M e^{\omega T}\}.
    \end{aligned}
    \end{equation*}
\end{proof}

\begin{lemma}\label{lem:adjlip}
    For every $g\in C([0,T]; Y_q^*)$, the unique solution of the equation 
    \begin{equation}\label{eq:psigl}
        - \psi' + A^*\psi = g, \quad \psi(T) = 0,
    \end{equation}
    is Lipschitz continuous on $[0,T]$ with values in $\DD^*$ and satisfies 
    \begin{equation*}
        \|\psi\|_{C^{0,1}(0,T;\DD^*)} \leq C'\, \|g\|_{C([0,T];Y_q^*)}
    \end{equation*}
    with a constant $C' >0$ independent of $g$.
\end{lemma}

\begin{proof}
    Since $X^* \subset D(A^*) = \{v\in \DD^*: \|A^* v\|_{\DD^*} < \infty\}$, there holds 
    according to \cite[Theorem~1.15.2(d)]{Tri78} that
    \begin{equation}\label{eq:triebel}
        Y_q^* = (\DD^*, X^*)_{1/q, q'}
        \embed (\DD^*, X^*)_{\theta, 1} 
        \embed (\DD^*, D(A^*))_{\theta, 1}
        \embed D((A^*)^\theta)
    \end{equation}
    for all $0 <\theta <1/q$.
    Therefore, $g\in C([0,T];D((A^*)^\theta))$ and consequently, by \cite[Theorem~4.3.6]{Paz83}, 
    the mild solution of \eqref{eq:psigl} given by 
    \begin{equation*}
        \psi(t) = \int_t^T e^{-(s - t)A^*}g(s)\,\d s 
    \end{equation*}
    is actually a classical solution, i.e., $\psi$ is continuous with values in $\DD^*$ on $[0,T)$,
    $\psi(t) \in X^*$ for all $t\in (0,T)$, and $\psi$ is continuously differentiable 
    with values in $\DD^*$ in $(0,T)$ and satisfies \eqref{eq:psigl} in all $t\in (0,T)$.
    Note that $A^*$ is the infinitesimal generator of an analytic semigroup, since it has maximal parabolic regularity 
    by Lemma~\ref{lem:maxparreg}\ref{it:maxparreg3}, cf.\ also \cite{Dor93}.
    This allows us to apply \cite[Theorem~6.13]{Paz83} to deduce
    \begin{equation*}
    \begin{aligned}
        \| A^* e^{-(s - t)A^*}g(s)\|_{\DD^*}
        &  =  \| (A^*)^{1-\theta} e^{-(s - t)A^*} (A^*)^\theta g(s) \|_{\DD^*} \\
        & \leq C\, |s-t|^{\theta - 1} \, \|(A^*)^\theta g(s) \|_{\DD^*}
        \leq C\, |s-t|^{\theta - 1} \, \| g(s)\|_{Y_q^*},
    \end{aligned}
    \end{equation*}        
    where we used \eqref{eq:triebel} for the last inequality. Since $\psi$ is a classical solution, it 
    satisfies \eqref{eq:psigl} pointwise in time and thus, we find for every $t\in (0,T)$ that
    \begin{equation*}
    \begin{aligned}
        \|\psi'(t)\|_{\DD^*} 
        & = \|A^* \psi(t) - g(t)\|_{\DD^*} \\
        & \leq \int_t^T \| A^* e^{-(s - t)A^*}g(s)\|_{\DD^*} \d s + \|g\|_{C([0,T];\DD^*)} \\
        & \leq C \int_t^T |s-t|^{\theta - 1}\,\d s\, \| g \|_{C([0,T]; Y_q^*)} + \|g\|_{C([0,T];\DD^*)} \\
        & \leq C ( T^\theta +1 ) \| g \|_{C([0,T]; Y_q^*)} .
    \end{aligned}
    \end{equation*}
    Thus, $\|\psi'(\cdot)\|_{\DD^*}$ is uniformly bounded on $(0,T)$ and consequently, 
    $\psi$ is Lipschitz continuous on the whole interval $[0,T]$ with Lipschitz constant $C(T^\theta + 1)$.
\end{proof}

    \section{On the lack of differentiability of objectives involving point evaluations}\label{sec:nodiff}
    
    This section is devoted to an elementary example demonstrating the lack of differentiability of 
    objectives involving point evaluations in time. 
    We consider the probably most simple setting where 
    \begin{equation*}
        \DD = X = \R, \quad A = f \equiv 0, \quad n = 2,
    \end{equation*}
    which leads to the following ODE as state equation
    \begin{equation*}
        y' = \chi_{[\tau_1, \tau_2)} \psi_1, \quad y(0) = 0.
    \end{equation*}
    We first investigate an objective of the form
    \begin{equation}\label{eq:ptobj1}
        J(y, \tau) := g(y(\tau_2))
    \end{equation}        
    with a continuously differentiable function $g:\R \to \R$.    
    For a given couple of switching times $\tau = (\tau_1, \tau_2)$, $\tau_1, \tau_2 \in \R$,
    the unique solution of the ODE reads as follows: If $\tau_1 \leq \tau_2$, $\tau_2 > 0$, and $\tau_1 < T$, 
    then the solution equals the continuous function
    \begin{equation*}
        y_\tau(t) = 
        \begin{cases}
            0, & t < \tau_1, \\
            \psi_1 (t - \max\{\tau_1,0\}), & \tau_1 \leq t < \tau_2, \\
            \psi_1(\tau_2 - \max\{\tau_1,0\}), & \tau_2 \leq t \leq T,
        \end{cases}
        \qquad t \in [0,T].
    \end{equation*}
    If else $\tau_1 > \tau_2$ or $\tau_1 \leq \tau_2 \leq 0$ or $T \leq \tau_1 \leq \tau_2$, 
    then the solution is simply $y_\tau \equiv 0$.
    Let now $\bar\tau\in (0,T)$ and $h_2 \in \R$ be arbitrary and set 
    \begin{equation*}
        \tau = (\tau_1, \tau_2) = (\bar\tau,\bar\tau), \quad 
        h = (h_1, h_2) = (0, h_2).
    \end{equation*}
    Then we obtain 
    \begin{equation*}
    \begin{aligned}
        \frac{y_{\tau + h}(\tau_2 + h_2) - y_\tau(\tau_2)}{\|h\|}
        = 
        \begin{cases}
            0, & h_2 \leq 0, \\
            \psi_1, & h_2 > 0
        \end{cases} 
    \end{aligned}
    \end{equation*}
    and the chain rule for Hadamard differentiable functions implies for the directional derivative of the reduced objective
    \begin{equation*}
        J'(y_\tau, \tau; h) = 
        \begin{cases}
            0, & h_2 \leq 0,\\        
            g'(0)\psi_1, & h_2 >0,
        \end{cases}
    \end{equation*}
    illustrating that a reduced objective of the form \eqref{eq:ptobj1} as a function of $\tau$ only
    is in general indeed not differentiable.
    
    Similarly, an objective of type
    \begin{equation}\label{eq:ptobj2}
        J(y, \tau) := g(y(\hat \tau))
    \end{equation}
    with $g$ as above and $\hat \tau \in[0,T]$ fixed is in general not differentiable.
    Indeed, if $\hat \tau > 0$, then we choose 
    \begin{equation*}
        \tau = (\tau_1, \tau_2) = (0, \hat \tau), \quad h = (0, h_2)
    \end{equation*}
    and a similar computation as above gives
    \begin{equation*}
        J'(y_\tau, \tau; h) = 
        \begin{cases}
            0, & h_2 \geq 0,\\        
            - g'(\psi_1\hat \tau)\psi_1, & h_2 < 0.
        \end{cases}
    \end{equation*}
    If $\hat\tau = 0$, then we set $\tau = (0,T)$ and $h = (h_1, 0)$ to obtain
    \begin{equation*}
        J'(y_\tau, \tau; h) = 
        \begin{cases}
            0, & h_1 \leq 0,\\        
            g'(0)\psi_1, & h_1 > 0
        \end{cases}
    \end{equation*}
    such that in both cases, $J$ is only directionally differentiable.

    \begin{remark}\label{rem:nodiff}
        The last example demonstrates that the replacement of $\chi_{[\tau_1, \tau_2)}$ by $\bar\chi_{[\tau_1, \tau_2)}$ 
        as defined in \eqref{eq:defbarchi} alone does not ensure differentiability, since $\tau_2 > \tau_1$ in this example and thus 
        $\chi_{[\tau_1, \tau_2)} = \bar\chi_{[\tau_1, \tau_2)}$.
        This illustrates that, even with this modification of $\chi_{[\tau_1, \tau_2)}$, the switching-point-to-state map is 
        in general not differentiable when considered as a mapping with values in $C([0,T])$. 
        This is in essence the reason for treating the state equation in weak form.
    \end{remark}

\end{appendix}


\bibliographystyle{plain}
\bibliography{Literaturverzeichnis}

\end{document}